\documentclass[reqno]{style}
\usepackage{boxedminipage,times,pslatex,amsmath, mathtools, graphicx, amssymb, paralist, cite, color}
\usepackage[all]{xypic}
\definecolor{e-mail}{rgb}{0,.40,.80}
\definecolor{reference}{rgb}{.20,.60,.22}
\definecolor{citation}{rgb}{0,.40,.80}
\usepackage[pdfstartview=FitH, colorlinks, linkcolor=reference, citecolor=citation, urlcolor=e-mail]{hyperref}

\newtheorem{thm}{Theorem}
\newtheorem{cor}[thm]{Corollary}
\newtheorem{lem}[thm]{Lemma}
\newtheorem{prop}[thm]{Proposition}
\newtheorem{claim}[thm]{Claim}
\theoremstyle{definition}
\newtheorem{defn}[thm]{Definition}
\theoremstyle{remark}
\newtheorem{rem}[thm]{Remark}
\numberwithin{thm}{section}
\theoremstyle{definition}

\theoremstyle{definition}

\theoremstyle{definition}
\numberwithin{equation}{section}

 \title[Computing difference-differential Galois groups of second-order equations]{Computation of the difference-differential Galois group and differential relations among solutions for a second-order linear difference equation}
\author{Carlos E. Arreche}
\email{cearrech@math.ncsu.edu}
\address{Mathematics Department, North Carolina State University, Raleigh, NC 27695}

\begin{document}


\begin{abstract}
We apply the difference-differential Galois theory developed by Hardouin and Singer to compute the differential-algebraic relations among the solutions to a second-order homogeneous linear difference equation of the form $ y(x+2)+a(x)y(x+1)+b(x)y(x)=0,$ where the coefficients $a(x),b(x)\in \bar{\mathbb{Q}}(x)$ are rational functions in $x$ with coefficients in $\bar{\mathbb{Q}}$. We develop algorithms to compute the difference-differential Galois group associated to such an equation, and show how to deduce the differential-algebraic relations among the solutions from the defining equations of the Galois group.
\end{abstract}

\maketitle

\section{Introduction} \label{intro-sec}

Consider a second-order homogeneous linear difference equation \begin{equation}\label{intro-eq} \sigma^2(y)+a\sigma(y)+by=0,\end{equation} whose coefficients $a,b\in\bar{\mathbb{Q}}(x)$, and where $\sigma$ denotes the $\bar{\mathbb{Q}}$-linear automorphism defined by $\sigma(x)=x+1$. We are motivated by the question: do the solutions of \eqref{intro-eq} satisfy any $\frac{d}{dx}$-algebraic equations over $\bar{\mathbb{Q}}(x)$? And if so, how can we compute all such differential-algebraic relations? We give complete answers to these questions as an application of the difference-differential Galois theory developed in \cite{hardouin-singer:2008}, which studies equations such as \eqref{intro-eq} from a purely algebraic point of view. This theory attaches a linear differential algebraic group $G$ (Definition~\ref{ldag-def}) to \eqref{intro-eq}, which group encodes all the difference-differential algebraic relations among the solutions to \eqref{intro-eq}. We develop an algorithm to compute $G$, and then show how the knowledge of $G$ leads to a concrete description of the sought difference-differential algebraic relations among the solutions.

The difference-differential Galois theory of \cite{hardouin-singer:2008} is a generalization of the difference Galois theory presented in \cite{vanderput-singer:1997}, where the Galois groups that arise are linear algebraic groups that encode the difference-algebraic relations among the solutions to a given linear difference equation. An algorithm to compute the difference Galois group $H$ associated to \eqref{intro-eq} by the theory of \cite{vanderput-singer:1997} is developed in \cite{hendriks:1998}. The computation of $G$ is more difficult than that of $H$, because there are many more linear differential algebraic groups than there are linear algebraic groups (more precisely, the latter are instances of the former), so identifying the correct difference-differential Galois group from among these possibilities requires additional work.

However, the difference Galois group $H$ serves as a close upper bound for the difference-differential Galois group $G$: it is shown in \cite{hardouin-singer:2008} that one can consider $G$ as a Zariski-dense subgroup of $H$ without loss of generality (see Proposition~\ref{dense} for a precise statement). In view of this fact, our strategy to compute $G$ is to first apply the algorithm of \cite{hendriks:1998} to compute $H$, and then compute the additional differential-algebraic equations (if any) that define $G$ as a subgroup of $H$. 

This strategy is reminiscent of the one begun in \cite{dreyfus:2011}, and concluded in \cite{arreche:2014a, arreche:2014b, arreche:2015}, to compute the parameterized differential Galois group for a second-order linear differential equation with differential parameters, where the results of \cite{kovacic:1986, baldassari-dwork:1979} are first applied to compute the classical (non-parameterized) differential Galois group for the differential equation, and one then computes the additional differential-algebraic equations, with respect to the parametric derivations, that define the parameterized differential Galois group inside the classical one. However, the computation of the difference-differential Galois group for \eqref{intro-eq} presents substantial new complications, which we describe below. 

Firstly, in the parameterized differential algorithm one first computes an associated unimodular differential Galois group as in \cite{dreyfus:2011, arreche:2014a}, and then recovers the original Galois group from this associated unimodular group and the change-of-variables data as in \cite{arreche:2014b, arreche:2015}. This reduction is not available in the difference-differential Galois theory, where it is not always possible to tensor away the effect of the determinant, so we must compute $G$ directly. An upshot of this is that the defining equations for $G$ obtained here are more explicit than those produced in \cite{arreche:2014b, arreche:2015} for the parameterized differential setting.

A second complication is that the inverse problem in the difference-differential Galois theory (that is: which linear differential algebraic groups arise as difference-differential Galois groups?) remains open, whereas in the parameterized differential Galois theory there is a complete answer to this question proved in \cite{dreyfus-thesis,mitschi-singer:2012b}. The knowledge of which linear differential algebraic groups arise as parameterized differential Galois groups is used systematically (though often implicitly) in the algorithms of \cite{dreyfus:2011, arreche:2014a, arreche:2014b,arreche:2015}. Partial progress on the inverse problem has been achieved recently in \cite{arreche-singer:2016}, where the authors characterize which groups can occur as difference-differential Galois groups of integrable or projectively integrable linear difference equations.

A third complication, which already occurs in the (non-differential) difference Galois theory of \cite{vanderput-singer:1997}, is that the Picard-Vessiot rings (Definition~\ref{pv-ring-def}---these are the analogues of the splitting fields encountered in classical Galois theory) in this theory are not always domains, which makes the application of the Galois correspondence (Theorem~\ref{correspondence}) more subtle in the present setting.

To put our work in context, let us mention some of the previous work that has been done on related problems. Algorithms to compute the Galois group for a second-order linear equation have been developed in the following cases: differential equations \cite{kovacic:1986}; difference equations \cite{hendriks:1998}; $q$-dilation equations \cite{hendriks:1997}; and Mahler equations \cite{roques:2015}. A general algorithm to compute the classical (non-differential) difference Galois group for a linear difference equation over $\bar{\mathbb{Q}}(x)$ of arbitrary order is presented in \cite{feng:2015b}. This algorithm can be considered as a difference analogue of the algorithm developed in \cite{hrushovski:2002} (see also \cite{feng:2015a}) to compute the classical (non-parameterized) differential Galois group for a linear differential equation of arbitrary order. Based on the results of \cite{hrushovski:2002}, algorithms were developed in \cite{minchenko-ovchinnikov-singer:2013a, minchenko-ovchinnikov-singer:2013b} to compute the parameterized differential Galois group of a linear differential equation of arbitrary order with differential parameters, subject to the condition that either the Galois group is reductive, or else that its maximal reductive quotient is differentially constant. An algorithm to compute the parameterized Galois group of a parameterized linear differential equation whose underlying differential operator is the composition of two completely reducible operators is presented in \cite{hardouin-ovchinnikov:2015}, and their results can be applied to compute difference-differential Galois groups in some of the cases that we consider here. Finally, algorithms to decide whether the solutions to certain linear $q$-dilation or Mahler equations are differentially independent are presented in \cite{dreyfus-hardouin-roques:2015}, and algorithms to decide whether the solutions to certain linear $q$-dilation difference equations satisfy additional functional equations are presented in \cite{dreyfus-hardouin-roques:2016}.

Let us now describe the contents of this work in more detail. In \S\ref{prelim-sec}, we summarize the difference-differential Galois theory of \cite{hardouin-singer:2008}, and prove some auxiliary results that will be used in the sequel. In \S\ref{hendriks-sec}, we summarize Hendriks' algorithm \cite{hendriks:1998} to compute the difference Galois group $H$ for \eqref{intro-eq}. In \S\ref{diagonalizable-sec}, we show how to compute the difference-differential Galois group $G$ for \eqref{intro-eq} when $H$ is diagonalizable in Proposition~\ref{galdiag}. In \S\ref{reducible-sec}, we show how to compute $G$ when $H$ is assumed to be reducible but nondiagonalizable in Proposition~\ref{wrat} and Proposition~\ref{bigut2}. The main theoretical result of \S\ref{reducible-sec} is Theorem~\ref{rad}, which states that the unipotent radical of $G$, a priori a linear differential algebraic group, is actually an algebraic group, which result is of independent interest. In \S\ref{dihedral-sec}, we compute $G$ in Proposition~\ref{dihedral-complete} as an application of Proposition~\ref{galdiag}, under the assumption that $H$ is irreducible and imprimitive. In \S\ref{large-sec}, we apply results from \cite{arreche-singer:2016} to compute $G$ in Theorem~\ref{g-large}, under the assumption that $H$ contains $\mathrm{SL}_2$. In \S\ref{relations-sec}, we show how to produce the difference-differential algebraic relations among the solutions of \eqref{intro-eq} from the knowledge of $G$. We conclude in \S\ref{examples-sec} by applying these results in some concrete examples.

\section{Preliminaries on difference-differential Galois theory} \label{prelim-sec}

We begin with a summary of the difference-differential Galois theory presented in \cite{hardouin-singer:2008}. We also prove some auxiliary results that will be useful in the following sections. Every field is assumed to be of characteristic zero.

\begin{defn} A $\sigma\delta$\emph{-ring} is a commutative ring $R$ with unit, equipped with an automorphism $\sigma$ and a derivation $\delta$ such that $\sigma\left(\delta(r)\right)=\delta\left(\sigma(r)\right)$ for every $r\in R$. A $\sigma\delta$-field is defined analogously. We write $$R^\sigma=\{r\in R \ | \ \sigma(r)=r\};\quad R^\delta=\{r\in R \ | \ \delta(r)=0\}; \quad\text{and}\quad R^{\sigma\delta}=R^\sigma\cap R^\delta,$$ and refer to these as the subrings of $\sigma$\emph{-constants}, $\delta$\emph{-constants}, and $\sigma\delta$\emph{-constants}, respectively.

A $\sigma\delta$-$R$-algebra is a $\sigma\delta$-ring $S$ equipped with a ring homomorphism $R\rightarrow S$ that commutes with both $\sigma$ and $\delta$. If $R$ and $S$ are fields, we also say that $S$ is a $\sigma\delta$-field extension of $R$. The notions of $\sigma$-$R$-algebra, $\delta$-$R$-algebra, $\sigma$-field extension, and $\delta$-field extension are defined analogously. If $z_1,\dots,z_n\in S$, we write $R\{z_1,\dots,z_n\}_\delta$ for the smallest $\delta$-$R$-subalgebra of $S$ that contains $z_1,\dots,z_n$; as $R$-algebras, we have $$R\{z_1,\dots,z_n\}_\delta=R[\{\delta^i(z_1),\dots,\delta^i(z_n) \ | \ i\in\mathbb{N}\}].$$ If $Z=(z_{ij})$ with $1\leq i,j\leq n$ is a matrix, we write $R=\{Z\}_\delta$ for $R\{z_{11},\dots,z_{1n},\dots,z_{n1},\dots,z_{nn}\}_\delta$.
\end{defn}

The main example of $\sigma\delta$-field that we will consider throughout most of this paper is $k=\bar{\mathbb{Q}}(x)$, where $\sigma$ denotes the $\bar{\mathbb{Q}}$-linear automorphism defined by $\sigma(x)= x+1$, and $\delta=\tfrac{d}{dx}$. Note that in this case $k^\sigma=k^\delta=\bar{\mathbb{Q}}$.

Suppose that $k$ is a $\sigma\delta$-field, and consider the matrix difference equation \begin{equation}\label{difeq1} \sigma(Y)=AY, \quad \text{where} \ A\in\mathrm{GL}_n(k).\end{equation}

\begin{defn} \label{pv-ring-def}A $\sigma\delta$\emph{-Picard-Vessiot ring} (or $\sigma\delta$\emph{-PV ring}) over $k$ for \eqref{difeq1} is a $\sigma\delta$-$k$-algebra $R$ such that:
\begin{enumerate}
\item[(i)] $R$ is a \emph{simple} $\sigma\delta$-ring, i.e., $R$ has no ideals, other than ${0}$ and $R$, that are invariant under both $\sigma$ and $\delta$;
\item[(ii)] there exists a matrix $Z\in\mathrm{GL}_n(R)$ such that $\sigma(Z)=AZ$; and
\item[(iii)] $R$ is differentially generated as a $\delta$-$k$-algebra by the entries of $Z$ and $1/\mathrm{det}(Z)$, i.e., $R=k\{Z,1/\mathrm{det}(Z)\}_\delta$.
\end{enumerate}
The matrix $Z$ is called a \emph{fundamental solution matrix} for \eqref{difeq1}. \end{defn}

Note that when $\delta=0$, this coincides with the definition of the $\sigma$-PV ring over $k$ for \eqref{difeq1} given in \cite[Def.~1.5]{vanderput-singer:1997}. In the usual Galois theory of difference equations presented in \cite{vanderput-singer:1997}, the existence and uniqueness of Picard-Vessiot rings up to $k$-$\sigma$-isomorphism is guaranteed by the assumption that $k^\sigma$ is algebraically closed (see \cite[\S1.1]{vanderput-singer:1997}). Analogously, in the difference-differential Galois theory developed in \cite{hardouin-singer:2008}, one needs to assume that $k^\sigma$ is $\delta$\emph{-closed} \cite{kolchin:1974, trushin:2010}.

\begin{defn} The ring of $\delta$-\emph{polynomials} in $n$ variables over a $\delta$-field $C$ is $$C\{Y_1,\dots,Y_n\}_\delta=C[\{\delta^i(Y_1),\dots,\delta^i(Y_n) \ | \ i\in\mathbb{N}\}],$$ the free $C$-algebra on the symbols $\delta^i(Y_j)$. We say $\mathcal{L}\in C\{Y_1,\dots,Y_n\}_\delta$ is a \emph{linear} $\delta$-\emph{polynomial} if it belongs to the $C$-linear span of the symbols $\delta^i(Y_j)$.

If $R$ is a $\delta$-$C$-algebra, we say that $z_1\dots,z_n\in R$ are \emph{differentially dependent} over $C$ if there exists a $\delta$-polynomial $0\neq P\in C\{Y_1,\dots,Y_n\}_\delta$ such that $P(z_1,\dots,z_n)=0$; otherwise we say that $z_1,\dots,z_n$ are $\delta$-\emph{independent} over $C$. When a single element $z\in R$ is $\delta$-independent (resp., $\delta$-dependent) over $k$, we also say that $z$ is $\delta$-\emph{transcendental} (resp., $\delta$-\emph{algebraic}) over $k$.

We say the $\delta$-field $C$ is $\delta$-\emph{closed} if, for any system of $\delta$-polynomial equations $$\{P_1=0,\dots,P_m=0 \ | \ P_i\in C\{Y_1,\dots,Y_n\}_\delta \ \ \text{for} \ \ 1\leq i\leq m\}$$ that admits a solution in $\tilde{C}^n$ for some $\delta$-field extension $\tilde{C}$ of $C$, there already exists a solution $C^n$.
\end{defn}

\begin{thm} \label{nnc}(Cf.~\cite[Prop.~2.4]{hardouin-singer:2008}) If $k^\sigma$ is $\delta$-closed, there exists a $\sigma\delta$-PV ring for \eqref{difeq1}, and it is unique up to $\sigma\delta$-$k$-isomorphism. Moreover, $R^\sigma=k^\sigma$.
\end{thm}

From now on, unless explicitly stated otherwise, we assume that $k$ is a $\sigma\delta$-field such that $k^\sigma$ is $\delta$-closed.

\begin{defn} The $\sigma\delta$\emph{-Galois group} of the $\sigma\delta$-PV ring $R$ for \eqref{difeq1} is the group of $\sigma\delta$-$k$-automorphisms of $R$: $$\mathrm{Gal}_{\sigma\delta}(R/k)=\{\gamma\in\mathrm{Aut}_{k\text{-alg}}(R) \ | \ \gamma\circ\sigma = \sigma\circ\gamma \ \text{and} \ \gamma\circ\delta = \delta\circ\gamma\}.$$\end{defn}

As in the usual (non-differential) Galois theory of difference equations \cite{vanderput-singer:1997}, the choice of fundamental solution matrix $Z\in\mathrm{GL}_n(R)$ defines a representation $\mathrm{Gal}_{\sigma\delta}(R/k)\hookrightarrow\mathrm{GL}_n(k^\sigma):\gamma\mapsto T_\gamma$, via $$\gamma(Z)=\begin{pmatrix} \gamma(z_{11}) & \cdots & \gamma(z_{1n}) \\ \vdots & & \vdots \\ \gamma(z_{n1}) & \cdots & \gamma(z_{nn}) \end{pmatrix} = \begin{pmatrix} z_{11} & \cdots & z_{1n} \\ \vdots & & \vdots \\ z_{n1} & \cdots & z_{nn} \end{pmatrix}\cdot T_\gamma.$$ A different choice of fundamental solution matrix $Z'\in\mathrm{GL}_n(R)$ defines a conjugate representation of $\mathrm{Gal}_{\sigma\delta}(R/k)$ in $\mathrm{GL}_n(k^\sigma)$.

\begin{defn} \label{equivalent-def} The systems $\sigma(Y)=AY$ and $\sigma(Y)=BY$ for $A,B\in\mathrm{GL}_n(k)$ are \emph{equivalent} if there exists a matrix $T\in\mathrm{GL}_n(k)$ such that $\sigma(T)AT^{-1}=B$. In this case, if $Z$ is a fundamental solution matrix for $\sigma(Y)=AY$, then $TZ$ is a fundamental solution matrix for $\sigma(Y)=BY$, and therefore the $\sigma\delta$-PV rings of $k$ for these systems defined by the choice of fundamental solution matrices $Z$ and $TZ$, and the associated representations of $\sigma\delta$-Galois groups in $\mathrm{GL}_n(k^\sigma)$, are the same.\end{defn}

\begin{defn} \label{ldag-def} Suppose that $C$ is a $\delta$-closed field. A \emph{linear differential algebraic group} over $C$ is a subgroup $G$ of $\mathrm{GL}_n(C)$ defined by (finitely many) $\delta$-polynomial equations in the matrix entries. We say that $G$ is $\delta$\emph{-constant} if $G$ is conjugate in $\mathrm{GL}_n(C)$ to a subgroup of $\mathrm{GL}_n(C^\delta)$.\end{defn}

The differential algebraic subgroups of the additive and multiplicative groups of $C$, which we denote respectively by $\mathbb{G}_a(C)$ and $\mathbb{G}_m(C)$, were classified in \cite[Prop.~11, Prop.~31 and its Corollary]{cassidy:1972}.

\begin{prop} \label{classification} If $G\leq \mathbb{G}_a(C)$ is a differential algebraic subgroup, then there exists a linear $\delta$-polynomial $\mathcal{L}\in C\{Y\}_\delta$ such that $$G=\{b\in\mathbb{G}_a(C) \ | \ \mathcal{L}(b)=0\}.$$

If $G\leq \mathbb{G}_m(C)$ is a differential algebraic subgroup, then either $G=\mu_\ell$, the group of $\ell^\text{th}$ roots of unity for some $\ell\in\mathbb{N}$, or else $\mathbb{G}_m(C^\delta)\subseteq G$, and there exists a linear $\delta$-polynomial $\mathcal{L}\in C\{Y\}_\delta$ such that $$G=\left\{a\in\mathbb{G}_m(C) \ \middle| \ \mathcal{L}\bigl(\tfrac{\delta a}{a}\bigr)=0\right\}.$$\end{prop}

\begin{thm} (Cf.~\cite[Thm.~2.6]{hardouin-singer:2008}) Suppose that $k^\sigma$ is $\delta$-closed, and that $R$ is a $\sigma\delta$-PV ring over $k$ for \eqref{difeq1}. Then $R$ is a reduced ring, and the choice of fundamental solution matrix $Z\in\mathrm{GL}_n(R)$ identifies $\mathrm{Gal}_{\sigma\delta}(R/k)$ with a linear differential algebraic subgroup of $\mathrm{GL}_n(k^\sigma)$.\end{thm}

As in \cite[p.~337]{hardouin-singer:2008}, we observe that if $R$ is a $\sigma\delta$-PV ring over $k$ for \eqref{difeq1}, and $K$ is the total ring of fractions of $R$, then any $\sigma\delta$-$k$-automorphism of $K$ must leave $R$ invariant, whence the group $\mathrm{Gal}_{\sigma\delta}(K/k)$ of such automorphisms coincides with $\mathrm{Gal}_{\sigma\delta}(R/k)$. The consideration of the total ring of fractions of $R$ is necessary to obtain the following Galois correspondence.

\begin{thm} \label{correspondence} (Cf.~\cite[Thm.~2.7]{hardouin-singer:2008}) Suppose that $k^\sigma$ is $\delta$-closed, and that $R$ is a $\sigma\delta$-PV ring over $k$ for \eqref{difeq1}. Denote by $K$ the total ring of fractions of $R$, and by $\mathcal{F}$ the set of $\sigma\delta$-rings $F$ such that $k\subseteq F \subseteq K$ and every non-zero divisor in $F$ is a unit in $F$. Let $\mathcal{G}$ denote the set of linear differential algebraic subgroups $H$ of $\mathrm{Gal}_{\sigma\delta}(K/k)$. There is a bijective correspondence $\mathcal{F}\leftrightarrow\mathcal{G}$ given by $$F \mapsto \mathrm{Gal}_{\sigma\delta}(K/F)=\{\gamma\in\mathrm{Gal}_{\sigma\delta}(K/k) \ | \ \gamma(r)=r, \ \forall r\in F\}; \quad \text{and} \quad H \mapsto K^H=\{r\in K \ | \ \gamma(r)=r, \ \forall \gamma\in H\}.$$\end{thm}

This implies in particular that an element $r\in K$ is left fixed by all of $\mathrm{Gal}_{\sigma\delta}(K/k)$ if and only if $r\in k$. The following result is a difference-differential analogue of \cite[Cor.~1.38]{vanderput-singer:2003}, and is proved similarly.

\begin{lem} \label{sfin} Suppose that $R$ is a $\sigma\delta$-PV ring over $k$ with total ring of fractions $K$. Let $G=\mathrm{Gal}_{\sigma\delta}(R/k)$ and suppose that $C=k^\sigma$ is $\delta$-closed. For any $z\in K$, the following properties are equivalent.
\begin{enumerate}
\item[(i)] $z\in R$.
\item[(ii)] The $C$-linear span $C\langle Gz\rangle$ of the orbit $Gz=\{\gamma(z) \ | \ \gamma\in G\}$ is finite dimensional as a $C$-vector space.
\item[(iii)] The $k$-linear span $k\langle \sigma^i(z)\rangle_{i\in \mathbb{Z}}$ of the orbit $\{\sigma^i(z) \ | \ i\in\mathbb{Z}\}$ is finite dimensional as a $k$-vector space.
\end{enumerate}\end{lem}

\begin{proof} (i)$\Rightarrow$(ii). By \cite[Prop.~6.24]{hardouin-singer:2008} and \cite[Rem.~4.37]{hardouin:2016}, $R$ is the coordinate ring of a $G$-torsor over $k$, which implies that there is a $\delta$-field extension $\tilde{k}$ of $k$ and a $G$-equivariant isomorphism (where $\tilde{k}$ is endowed with the trivial $G$-action) of $\tilde{k}$-algebras $\tilde{k}\otimes_kR\simeq \tilde{k}\otimes_C C\{G\}$, where $C\{G\}$ denotes the $\delta$-Hopf algebra of coordinate functions on $G$ \cite[Def.~5]{ovchinnikov:2008}. It follows from \cite[Prop.~10]{cassidy:1972} (see also \cite[Prop.~2.2.3]{minchenko-ovchinnikov-singer:2013b}) that the $G$-orbit of any element of $C\{G\}$ spans a finite dimensional vector space over $C$. This property is inherited by $\tilde{k}\otimes_CC\{G\}$ and also by $R$.

(ii)$\Rightarrow$(iii). It follows from the proof of \cite[Lem.~A.6]{hendriks-singer:1999}, with our $K$ replacing the $R$ in the statement of that Lemma, that $C\langle Gz\rangle$ is the solution space of a homogeneous linear difference equation over $k$.

(iii)$\Rightarrow$(i). Let $W=k\langle \sigma^i(z)\rangle_{i\in \mathbb{Z}}$, and let $I$ be the ideal of $R$ consisting of elements $a\in R$ such that $aW\subset R$. Since $W$ is $k$-finite dimensional, $I$ is non-zero, because for any finite collection $r_1,\dots,r_m\in K$ there exists a nonzero $a\in R$ such that $ar_1,\dots, ar_m\in R$. We claim that $I$ is a $\sigma$-ideal of $R$: if $a\in I$ and $w\in W$, then $\sigma^{-1}(\sigma(a)w)=a\sigma^{-1}(w)\in R$. Since $\sigma$ is an automorphism of $R$, this implies that $\sigma(a)w\in R$, and therefore $\sigma(a)\in I$. By \cite[Cor.~6.22]{hardouin-singer:2008}, $R$ is $\sigma$-simple, and therefore $1\in I$, which implies that $W\subset R$ and in particular $z\in R$.
\end{proof}

The following result relates the $\sigma\delta$-PV rings and $\sigma\delta$-Galois groups of \cite{hardouin-singer:2008} to the $\sigma$-PV rings and $\sigma$-Galois groups considered in \cite{vanderput-singer:1997, hendriks:1998}.

\begin{prop} \label{dense} (Cf.~\cite[Prop.~2.8]{hardouin-singer:2008}) Assume $k^\sigma$ is $\delta$-closed. Let $R$ be a $\sigma\delta$-PV ring over $k$ for \eqref{difeq1} with fundamental solution matrix $Z\in\mathrm{GL}_n(R)$, and let $S=k[Z,1/\mathrm{det}(Z)]\subset R$. Then:
\begin{enumerate}
\item[(i)] $S$ is a $\sigma$-PV ring over $k$ for \eqref{difeq1}; and
\item[(ii)] $\mathrm{Gal}_{\sigma\delta}(R/k)$ is Zariski-dense in the $\sigma$-Galois group $\mathrm{Gal}_\sigma(S/k)$.
\end{enumerate}
\end{prop}

The following result characterizes those difference equations whose $\sigma\delta$-Galois groups are $\delta$-constant.

\begin{prop} \label{dconst} (Cf.~\cite[Prop.~2.9]{hardouin-singer:2008}) Let $R$ be a $\sigma\delta$-PV ring over $k$ for $\sigma(Y)=AY$, where $A\in\mathrm{GL}_n(k)$ and $k^\sigma$ is $\delta$-closed. Then $\mathrm{Gal}_{\sigma\delta}(R/k)$ is a $\delta$-constant linear differential algebraic group if and only if there exists a matrix $B\in\mathfrak{gl}_n(k)$ such that $$\sigma(B)=ABA^{-1} +\delta(A)A^{-1}.$$ In this case, there exists a fundamental solution matrix $Z\in\mathrm{GL}_n(R)$ that satisfies the system $$ \begin{cases} \sigma(Z)  =AZ; \\ \delta(Z) = BZ.\end{cases}$$
\end{prop}

The following result is proved in \cite[Prop.~3.1]{hardouin-singer:2008}.

\begin{prop}\label{diffsum} Let $R$ be a $\sigma\delta$-$k$-algebra with $R^\sigma=k^\sigma$. Let $b_1,\dots,b_m\in k$ and $z_1,\dots,z_m\in R$ satisfy $$\sigma(z_i)-z_i=b_i; \quad i=1,\dots,m.$$ Then $z_1,\dots,z_m$ are differentially dependent over $k$ if and only if there exists a nonzero homogeneous linear differential polynomial $\mathcal{L}(Y_1,\dots,Y_m)$ with coefficients in $k^\sigma$ and an element $f\in k$ such that $$\mathcal{L}(b_1,\dots,b_m)=\sigma(f)-f.$$\end{prop}

The following notion defined in \cite[Def.~2.3]{chen-singer:2012} will be crucial in several proofs in this paper.

\begin{defn} (Discrete residues) Consider the $\sigma$-field $k=C(x)$, where $k^\sigma=C$ is algebraically closed and $\sigma(x)=x+1$. For any $\beta\in C$, we call the subset $[\beta]=\beta+\mathbb{Z}\subset  C$ the $\mathbb{Z}$-orbit of $\alpha$ in $C$. Any $f\in k$ can be decomposed into the form $$f=p+\sum_{i=1}^m\sum_{j=1}^{n_i}\sum_{\ell=0}^{d_{i,j}}\frac{\alpha_{i,j,\ell}}{(x-(\beta_i+\ell))^j},$$ where $p\in C[x]$, $m,n_i,d_{i,j}\in\mathbb{N}$, $\alpha_{i,j,\ell},\beta_i\in C$, and the $\beta_i$ belong to different $\mathbb{Z}$-orbits. We define the \emph{discrete residue} of $f$ at the $\mathbb{Z}$-orbit $[\beta_i]$ of multiplicity $j$ (with respect to $x$) as: $$\mathrm{dres}_x(f,[\beta_i],j)=\sum_{\ell=0}^{d_{i,j}}\alpha_{i,j,\ell}.$$
\end{defn}

The usefulness of the notion of discrete residue stems from the following result.

\begin{prop}\label{indsum}(Cf.~\cite[Prop.~2.5]{chen-singer:2012}) Let $k=C(x)$ be the $\sigma$-field defined by $k^\sigma=C$ and $\sigma(x)=x+1$, and let $f,g\in C[x]$ be non-zero, relatively prime polynomials. There exists $h\in k$ such that $\sigma(h)-h=f/g$ if and only if $\mathrm{dres}_x(f/g,[\beta],j)=0$ for every multiplicity $j$ and every $\mathbb{Z}$-orbit $[\beta]$ such that $\beta\in \bar{C}$ satisfies $g(\beta)=0$.
\end{prop}

The following result is a variant of \cite[Cor.~3.3]{hardouin-singer:2008} and \cite[Thm.~4.2]{hardouin:2008}.

\begin{cor} \label{diffprod} Let $C$ be a $\delta$-closed field, and consider $k=C(x)$ as a $\sigma\delta$-field by letting $\delta(x)=1$ and $\sigma$ be the $C$-linear automorphism of $k$ given by $\sigma(x)=x+1$. Let $R$ be a $\sigma\delta$-$k$-algebra with $R^\sigma=k^\sigma=C$. Let $a_1,\dots,a_m\in C^\delta(x)$ and $z_1,\dots,z_m\in R$, all nonzero, such that $$\sigma(z_i)=a_iz_i; \quad i=1,\dots,m.$$ Then $z_1,\dots,z_m$ are differentially dependent over $k$ if and only if there exist integers $n_1,\dots,n_m\in\mathbb{Z}$, not all zero, and an element $f\in k$ such that \begin{equation*} n_1\frac{\delta (a_1)}{a_1}+\dots+n_m\frac{\delta (a_m)}{a_m}=\sigma(f)-f.\end{equation*}\end{cor}

\begin{proof} Since $\sigma(\frac{\delta(z_i)}{z_i})=\frac{\delta(z_i)}{z_i}+\frac{\delta(a_i)}{a_i}$ for each $i=1,\dots,m$, Proposition~\ref{diffsum} implies that the $z_i$ are differentially dependent over $k$ if and only if there exists an element $f\in k$ and a nonzero linear differential polynomial $$\mathcal{L}(Y_1,\dots,Y_m)=\sum_{i=1}^m\sum_{j=0}^{r_i} c_{i,j}\delta^jY_i, \quad c_{i,j}\in C,$$ such that \begin{equation}\label{ord1rel}g=\mathcal{L}\left(\frac{\delta(a_1)}{a_1},\dots,\frac{\delta(a_m)}{a_m}\right)=\sigma(f)-f.\end{equation}

Let $r=\mathrm{max}\{r_i \ | \ c_{i,r_i}\neq 0 \ \text{for some} \ i\}$. For each $\beta\in C$, it follows from \eqref{ord1rel} and Proposition~\ref{indsum} that $$\mathrm{dres}_x(g,[\beta],r+1)=(-1)^rr!\sum_{i=1}^m c_{i,r}\mathrm{dres}_x\left(\frac{\delta(a_i)}{a_i},[\beta],1\right)=0.$$ Since $\mathrm{dres}_x(\tfrac{\delta(a_i)}{a_i},[\beta],1)\in\mathbb{Z}$ for each $\beta\in C$ and each $i=1,\dots,m$, we may take the $c_{i,r}=n_i$ to be integers, and the conclusion follows from another application of Proposition~\ref{indsum}.
\end{proof}

The following result will be used later to assume that the $\sigma\delta$-Galois groups that we wish to compute are connected, after replacing $\sigma$ with $\sigma^t$ for some $t\in\mathbb{N}$.

\begin{prop} \label{idemp} Let $R$ be a $\sigma\delta$-PV ring over $k$ for \eqref{difeq1}, where $k=C(x)$, $\delta(x)=1$, $C$ is $\delta$-closed, and $\sigma$ is the $C$-linear automorphism of $k$ given by $\sigma:x\mapsto x+1$. There exist idempotents $e_0,\dots,e_{t-1}\in R$ such that:
\begin{enumerate}
\item[(i)] $R=R_0\oplus\dots\oplus R_{t-1}$, where $R_i= e_i R$;
\item[(ii)] the action of $\sigma$ permutes the set $\{ R_0,\dots,R_{t-1}\}$ transitively, and each $R_i$ is left invariant by $\sigma^t$;
\item[(iii)] each $R_i$ is a domain, and is a $\sigma^t\delta$-PV ring over $k$ for $\sigma^t(Y)=A_tY$, where $A_t=\sigma^{t-1}(A)\dots\sigma(A)A$;
\item[(iv)] the $\sigma^t\delta$-Galois group $\mathrm{Gal}_{\sigma^t\delta}(R_0/k)$ is identified with the connected component of the identity $\mathrm{Gal}_{\sigma\delta}(R/k)^0$.
\end{enumerate}
\end{prop}

\begin{proof} Items (i) and (ii) are proved in \cite[Lem.~6.8(1,3)]{hardouin-singer:2008}. Item (iii) follows from \cite[Lem.~6.8(3)]{hardouin-singer:2008}, Proposition~\ref{dense}, and \cite[Lem.~1.26]{vanderput-singer:1997} (although this last result is proved in the non-differential setting, the same proof works here \emph{mutatis mutandis}). By \cite[Thm.~2.6(2)]{hardouin-singer:2008}, $\mathrm{Gal}_{\sigma^t\delta}(R_0/k)$ is connected, since $R_0$ is a domain. By \cite[Cor.~1.17]{vanderput-singer:1997}, there is an exact sequence $$0\longrightarrow \mathrm{Gal}_{\sigma^t\delta}(R_0/k)\longrightarrow\mathrm{Gal}_{\sigma\delta}(R/k)\longrightarrow\mathbb{Z}/t\mathbb{Z}\longrightarrow 0,$$ which proves (iv) (cf. \cite[Prop.~4.5 and its proof]{vanderput-singer:1997}).
\end{proof}

\section{Hendriks' algorithm} \label{hendriks-sec}

In this section, we summarize the results of \cite{hendriks:1998} that we will need in our algorithm. From now on, we restrict our attention to equations of the form \begin{equation}\label{difeq} \sigma^2(y) + a\sigma(y) + by = 0,\end{equation} where $a,b\in\bar{\mathbb{Q}}(x)$, and $\sigma$ is the $\bar{\mathbb{Q}}$-linear automorphism of $\bar{\mathbb{Q}}(x)$ defined by $\sigma(x)=x+1$. We consider $\bar{\mathbb{Q}}(x)$ as a $\sigma\delta$-field by setting $\delta=\tfrac{d}{dx}$. The matrix equation corresponding to \eqref{difeq} is \begin{equation}\label{mateq} \sigma(Y)=\begin{pmatrix} 0 & 1 \\ -b & -a\end{pmatrix}Y.\end{equation} In order to apply the theory of \cite{hardouin-singer:2008} to study \eqref{difeq}, we will consider \eqref{difeq} as a difference equation over $k=C(x)$, where $C$ is a $\delta$-closed field extension of $(\bar{\mathbb{Q}},\delta)$ (the existence of such a $C$ is guaranteed by \cite{kolchin:1974, trushin:2010}), and the $\sigma\delta$-structure of $k$ extends that of $\bar{\mathbb{Q}}(x)$: $\sigma$ is the $C$-linear automorphism of $k$ defined by $\sigma(x)=x+1$, and $\delta(x)=1$.

Let $R$ denote a $\sigma\delta$-PV ring over $k$ for \eqref{mateq} with fundamental solution matrix $Z$, and let $k[Z,1/\mathrm{det}(Z)]=S\subseteq R$ denote a $\sigma$-PV ring for \eqref{mateq} (cf.~Proposition~\ref{dense}). We will write $G=\mathrm{Gal}_{\sigma\delta}(R/k)$ and $H=\mathrm{Gal}_\sigma(S/k)$. In view of Proposition~\ref{dense}, in order to compute $G$ we will first apply the results of \cite{hendriks:1998} to compute $H$, and then compute the additional $\delta$-algebraic equations that define $G$ as a subgroup of $H$.

The algorithm developed in \cite{hendriks:1998} to compute $H$ proceeds as follows. We first decide whether there exists a solution $u\in\bar{\mathbb{Q}}(x)$ to the \emph{Riccati equation} \begin{equation}\label{ric1} u\sigma(u)+au+b=0.\end{equation} If such a solution $u$ exists, then $H$ is conjugate to an algebraic subgroup of $$\mathbb{G}_m(C)^2\ltimes\mathbb{G}_a(C)\simeq\left\{\begin{pmatrix} \alpha & \beta \\ 0 & \lambda \end{pmatrix} \ \middle| \ \alpha,\beta,\lambda\in C, \ \alpha\lambda\neq 0\right\}.$$ Moreover, if there exist at least two distinct solutions $u, v\in\bar{\mathbb{Q}}(x)$ to \eqref{ric1}, then $H$ is conjugate to an algebraic subgroup of $$\mathbb{G}_m(C)^2\simeq\left\{\begin{pmatrix} \alpha & 0 \\ 0 & \lambda\end{pmatrix} \ \middle| \ \alpha,\lambda\in C, \ \alpha\lambda\neq 0\right\};$$ and if there are at least three distinct solutions in $\bar{\mathbb{Q}}(x)$ to \eqref{ric1}, then $H$ is an algebraic subgroup of $$\mathbb{G}_m(C)\simeq\left\{\begin{pmatrix} \alpha & 0 \\ 0 & \alpha\end{pmatrix} \ \middle| \ \alpha\in C, \ \alpha\neq 0\right\}.$$

If there is no solution $u\in\bar{\mathbb{Q}}(x)$ to the Riccati equation \eqref{ric1}, we then attempt to find $T\in\mathrm{GL}_2(\bar{\mathbb{Q}}(x))$ and $r\in \bar{\mathbb{Q}}(x)$ such that \begin{equation}\label{impequiv}\sigma(T)\begin{pmatrix} 0 & 1 \\ -b & -a\end{pmatrix}T^{-1}=\begin{pmatrix} 0 & 1 \\ -r & 0\end{pmatrix}.\end{equation} If $a=0$ already, then we may take $T=\left(\begin{smallmatrix}1 & 0 \\ 0 & 1\end{smallmatrix}\right)$ and $r=b$. If $a\neq 0$, we then attempt to find a solution $e\in\bar{\mathbb{Q}}(x)$ to the Riccati equation \begin{equation}\label{ric2} e\sigma^2(e) + \bigl(\sigma^2(\tfrac{b}{a}) - \sigma(a) + \tfrac{\sigma(b)}{a}\bigr)e + \tfrac{\sigma(b)b}{a^2}=0.\end{equation} If there exists such a solution $e\in\bar{\mathbb{Q}}(x)$ to \eqref{ric2}, then it is proved in \cite[Thm.~4.6]{hendriks:1998} that there exists a matrix $T\in\mathrm{GL}_2(\bar{\mathbb{Q}}(x))$ such that \eqref{impequiv} is satisfied with $r=-a\sigma(a)+\sigma(b)+a\sigma^2(\tfrac{b}{a})+a\sigma^2(e)$, and $H$ is conjugate to an algebraic subgroup of \begin{equation}\label{dihedral}\{\pm1\}\ltimes\mathbb{G}_m(C)^2\simeq\left\{\begin{pmatrix} \alpha & 0 \\ 0 & \lambda\end{pmatrix} \ \middle| \ \alpha,\lambda\in C, \ \alpha\lambda\neq 0\right\} \cup\left\{\begin{pmatrix} 0 & \beta \\ \epsilon & 0\end{pmatrix} \ \middle| \ \beta,\epsilon\in C, \ \beta\epsilon\neq 0\right\}.\end{equation}

Finally, if $a\neq 0$ and neither \eqref{ric1} nor \eqref{ric2} admits a solution in $\bar{\mathbb{Q}}(x)$, then $\mathrm{SL}_2(C)\subseteq H$, and we compute $H$ as in \cite[\S4.4]{hendriks:1998}.

\begin{rem}[(Descent from $k$ to $\bar{\mathbb{Q}}(x)$)] \label{c-ok} The application of the algorithm of \cite{hendriks:1998} to compute the $\sigma$-Galois group of \eqref{difeq} over $k$, rather than over $\bar{\mathbb{Q}}(x)$, requires some justification. The point is that the explicit steps involved in this algorithm require finding all solutions in $k$ to some polynomial $\sigma$-equations defined over $\bar{\mathbb{Q}}(x)$. If there exist solutions to these $\sigma$-equations in $k$, then there also exist solutions in $\bar{\mathbb{Q}}(x)$.

This follows from an elementary argument: suppose that a given polynomial $\sigma$-equation over $\bar{\mathbb{Q}}(x)$ admits a solution $\frac{p}{q}\in C(x)$, where $p=a_nx^n+\dots+a_ix+a_0$ and $q=b_mx^m+\dots+b_1x+b_0$. This is equivalent to the coefficients $a_i$ and $b_j$ satisfying a system of polynomial equations defined over $\bar{\mathbb{Q}}$, which defines an affine algebraic variety $V$ over $\bar{\mathbb{Q}}$. Since $\bar{\mathbb{Q}}$ is algebraically closed, if $V(C)$ is nonempty then $V(\bar{\mathbb{Q}})$ is also nonempty. \end{rem}

\section{Diagonalizable groups} \label{diagonalizable-sec} We recall the notation introduced in the previous sections: $k=C(x)$, where $C$ is a $\delta$-closure of $\bar{\mathbb{Q}}$, $\sigma$ denotes the $C$-linear automorphism of $k$ defined by $\sigma(x)=x+1$, and $\delta(x)=1$. Let us first suppose that there exist at least two distinct solutions $u,v\in\bar{\mathbb{Q}}(x)$ to the Riccati equation \eqref{ric1}. Then it is shown in \cite[proof of Thm.~4.2(3)]{hendriks:1998} that \eqref{mateq} is equivalent to $$\sigma(Y)=\begin{pmatrix} u & 0 \\ 0 & v\end{pmatrix}Y.$$ In this case, we compute $G$ with the following result.

\begin{prop} \label{galdiag} Assume that $u,v\in \bar{\mathbb{Q}}(x)$ are both different from $0$, and let $P$ be the $\sigma\delta$-PV ring over $k$ corresponding to the system$$\sigma(Y)=\begin{pmatrix} u & 0 \\ 0 & v\end{pmatrix}Y.$$ Then $\mathrm{Gal}_{\sigma\delta}(P/k)$ is the subgroup of \begin{equation}\label{galdiag-eq}\mathbb{G}_m(C)^2=\left\{\begin{pmatrix} \alpha & 0 \\ 0 & \lambda\end{pmatrix} \ \middle| \ \alpha,\lambda\in C, \ \alpha\lambda\neq 0 \right\}\end{equation} defined by the following conditions on $\alpha$ and $\lambda$.

\begin{enumerate}
\item[(i)] There exist: nonzero elements $f,g\in \bar{\mathbb{Q}}(x)$, a primitive $m$-th root of unity $\zeta_m$ such that $u=\zeta_m\frac{\sigma(f)}{f}$, and a primitive $n$-th root of unity $\zeta_n$ such that $v=\zeta_n\frac{\sigma(g)}{g}$, if and only if $\alpha^m=1$, $\lambda^n=1$, and $\alpha^{eg_m}\lambda^{g_n}=1$ for some positive integer $e$ with $\mathrm{gcd}(m,n,e)=1$, where $g_m=\frac{m}{\mathrm{gcd}(m,n)}$ and $g_n=\frac{n}{\mathrm{gcd}(m,n)}$.
\item[(ii)] There exist: integers $m$ and $n$, not both zero, and a nonzero element $f\in \bar{\mathbb{Q}}(x)$ such that $u^mv^n=\frac{\sigma(f)}{f}$, if and only if $\alpha^m\lambda^n=1$.
\item[(iii)] There exists an element $f\in \bar{\mathbb{Q}}(x)$ such that $\frac{\delta(u)}{u}=\sigma(f)-f$ (resp., such that $\frac{\delta(v)}{v}=\sigma(f)-f$) if and only if $\delta(\alpha)=0$ (resp., $\delta(\lambda)=0$).
\item[(iv)] There exist relatively prime integers $m$ and $n$ and an element $f\in \bar{\mathbb{Q}}(x)$ such that $m\frac{\delta(u)}{u}+n\frac{\delta(v)}{v}=\sigma(f)-f$ if and only if $\delta(\alpha^m\lambda^n)=0$.
\item[(v)] If none of the conditions above are satisfied, then $\mathrm{Gal}_{\sigma\delta}(P/k)=\mathbb{G}_m(C)^2$.
\end{enumerate}
\end{prop}

\begin{proof} We begin by observing that, if we can find elements $f,g\in k$ witnessing the relations in items (i)--(iv), then we may take $f,g\in \bar{\mathbb{Q}}(x)$, since $u,v\in\bar{\mathbb{Q}}(x)$ (cf. \cite[Lem.~2.4, Lem.~2.5]{hardouin:2008} and Remark~\ref{c-ok}). Note that by Theorem~\ref{nnc}, $P^\sigma=C$.

Items (i) and (ii) are proved in \cite[Lem.~4.4]{hendriks:1998}.

Let $y_1, y_0\in P$ be nonzero elements such that $\sigma(y_1)=uy_1$ and $\sigma(y_0)=vy_0$, and let $g_i=\frac{\delta(y_i)}{y_i}$ for $i=0,1$. Then $$\sigma(g_1)-g_1=\frac{\delta(u)}{u} \qquad \text{and}\qquad \sigma(g_0)-g_0=\frac{\delta(v)}{v}.$$ On the other hand, for every $\gamma\in\mathrm{Gal}_{\sigma\delta}(P/k)$ we have that $\gamma(y_1)=\alpha_\gamma y_1$ and $\gamma(y_0)=\lambda_\gamma y_0$, and therefore $$\gamma(g_1)=g_1+\frac{\delta(\alpha_\gamma)}{\alpha_\gamma} \qquad \text{and}\qquad \gamma(g_0)=g_0+\frac{\delta(\lambda_\gamma)}{\lambda_\gamma}.$$ Hence, it follows from Theorem~\ref{correspondence} that $\delta(\alpha_\gamma)=0$ (resp., $\delta(\lambda_\gamma)=0$) for every $\gamma\in\mathrm{Gal}_{\sigma\delta}(R/k)$ if and only if $g_1\in k$ (resp., $g_0\in k$), which implies (iii).

Moreover, for any integers $m,n\in\mathbb{Z}$ and any $\gamma\in\mathrm{Gal}_{\sigma\delta}(P/k)$ we have that $$\gamma(mg_1+ng_0)=mg_1+ng_0+m\frac{\delta(\alpha_\gamma)}{\alpha_\gamma}+n\frac{\delta(\lambda_\gamma)}{\lambda_\gamma}=mg_1+ng_0+\frac{\delta(\alpha_\gamma^m\lambda_\gamma^n)}{\alpha_\gamma^m\lambda_\gamma^n}.$$ Hence, Theorem~\ref{correspondence} implies that $\delta(\alpha_\gamma^m\lambda_\gamma^n)=0$ for every $\gamma\in\mathrm{Gal}_{\sigma\delta}(P/k)$ if and only if $mg_1+ng_0\in k$, and since $$\sigma(mg_1+ng_0)-(mg_1+ng_0)=m\frac{\delta(u)}{u}+n\frac{\delta(v)}{v},$$ this implies (iv).

Part (v) follows from Corollary~\ref{diffprod}.
\end{proof}

\begin{rem} \label{diagonalizable-rem} To compute the difference-differential Galois group $G$ for \eqref{difeq} when there exist at least two distinct solutions $u,v\in\bar{\mathbb{Q}}(x)$ to the Riccati equation \eqref{ric1}, we proceed as follows. First, compute the \emph{discrete residues} $r_{u,[d]}=\mathrm{dres}_x(\frac{\delta(u)}{u},[d],1)$ and $r_{v,[d]}=\mathrm{dres}_x(\frac{\delta(v)}{v},[d],1)$ at each $\mathbb{Z}$-orbit $[d]$ for $d\in \bar{\mathbb{Q}}$. Observe that $r_{u,[d]},r_{v,[d]}\in\mathbb{Z}$ for every $[d]$.

The following cases all refer to Proposition~\ref{galdiag}. By \cite[Lem.~2.1]{vanderput-singer:1997}, case (i) occurs when: $r_{u,[d]}=r_{v,[d]}=0$ for every $[d]$, $u(\infty)=\zeta_m$, and $v(\infty)=\zeta_n$; the integer $e$ in this case is the smallest positive integer such that $\zeta_m^{eg_m}\zeta_n^{g_n}=1$. By \cite[Lem.~2.1]{vanderput-singer:1997}, case (ii) occurs when: $mr_{u,[d]}+nr_{v,[d]}=0$ for every $[d]$ simultaneously, and $(u^m\cdot v^n)(\infty)=1$. By Proposition~\ref{indsum}, case (iii) occurs when $r_{u,[d]}=0$ (resp., when $r_{v,[d]}=0$) for every $[d]$. By Proposition~\ref{indsum}, case (iv) occurs when there exist nonzero integers $m,n\in\mathbb{Z}$ such that $mr_{u,[d]}+nr_{v,[d]}=0$ for every $[d]$ simultaneously.
\end{rem}

\section{Reducible groups} \label{reducible-sec}
We recall the notation introduced in the previous sections: $k=C(x)$, where $C$ is a $\delta$-closure of $\bar{\mathbb{Q}}$, $\sigma$ denotes the $C$-linear automorphism of $k$ defined by $\sigma(x)=x+1$, and $\delta(x)=1$.

We now proceed to define the additional notation that we will use throughout this section. We will assume that there exists exactly one solution $u\in\bar{\mathbb{Q}}(x)$ to the Riccati equation \eqref{ric1}, so that the $\sigma$-Galois group $H$ for \eqref{difeq} is reducible but not completely reducible, and the difference operator implicit in \eqref{difeq} factors as $$\sigma^2+a\sigma+b=(\sigma-\tfrac{b}{u})\circ(\sigma - u),$$ as we saw in \S\ref{hendriks-sec}.
This means that there is a $C$-basis of solutions $\{y_1,y_2\}$ in any $\sigma\delta$-PV ring $R$ for \eqref{difeq} such that $y_1,y_2\neq 0$ satisfy $\sigma(y_1)=uy_1$ and $\sigma(y_2)-uy_2=y_0$, where $y_0\neq 0$ satisfies $\sigma(y_0)=\tfrac{b}{u}y_0$. A fundamental solution matrix for \eqref{mateq} is given by \begin{equation}\label{fundsol}\begin{pmatrix} y_1 & y_2 \\ \sigma(y_1) & \sigma(y_2)\end{pmatrix}=\begin{pmatrix} y_1 & y_2 \\ uy_1 & uy_2+y_0\end{pmatrix}.\end{equation}

If we now let $A=\left(\begin{smallmatrix} 0 & 1 \\ -b & -a\end{smallmatrix}\right)$, $T=\left(\begin{smallmatrix} 1-u & 1 \\ -u & 1\end{smallmatrix}\right)$, and $v=\tfrac{b}{u}=-\sigma(u)-a$ (since $u$ satisfies \eqref{ric1}), we have that \begin{align*}\sigma(T)AT^{-1}&=\begin{pmatrix} 1-\sigma(u) & 1 \\ -\sigma(u) & 1\end{pmatrix}\begin{pmatrix} 0 & 1 \\ -b & -a \end{pmatrix}\begin{pmatrix}1 & -1 \\ u & 1-u\end{pmatrix} \\ &=\begin{pmatrix}  u -u\sigma(u)-au-b & 1-u-\sigma(u)-a+u\sigma(u)+au+b \\ -u\sigma(u)-au-b & -\sigma(u)-a + u\sigma(u)+au+b\end{pmatrix} \\ &=\begin{pmatrix}u & 1-u+v \\ 0 & v\end{pmatrix}=B.\end{align*} Therefore, the systems \eqref{mateq} and $\sigma(Y)=BY$ are equivalent (in the sense of Definition~\ref{equivalent-def}), and a fundamental solution matrix for the latter system is given by $$\begin{pmatrix} 1-u & 1 \\ -u & 1\end{pmatrix}\begin{pmatrix} y_1 & y_2 \\ uy_1 & uy_2+y_0\end{pmatrix}=\begin{pmatrix} y_1 & y_2+y_0 \\ 0 & y_0\end{pmatrix}.$$

For any $\gamma\in H$, the $\sigma$-Galois group for \eqref{difeq}, we have that \begin{equation}\label{h-emb} \gamma\begin{pmatrix} y_1 & y_2+y_0 \\ 0 & y_0\end{pmatrix}=\begin{pmatrix} y_1 & y_2+y_0 \\ 0 & y_0\end{pmatrix}\begin{pmatrix} \alpha_\gamma & \beta_\gamma \\ 0 & \lambda_\gamma\end{pmatrix} = \begin{pmatrix} \alpha_\gamma y_1 & \beta_\gamma y_1+\lambda_\gamma y_2+\lambda_\gamma y_0 \\ 0 & \lambda_\gamma y_0\end{pmatrix},\end{equation} and therefore the action of $H$ on the solutions is defined by \begin{equation}\label{h-action}\gamma(y_1)=\alpha_\gamma y_1; \qquad \gamma(y_0)=\lambda_\gamma y_0; \qquad \text{and}\qquad \gamma(y_2)=\lambda_\gamma y_2 + \beta_\gamma y_1.\end{equation} It will be convenient to define the auxiliary elements \begin{equation} \label{zw-def} w=\frac{y_0}{uy_1}  \qquad \text{and}\qquad z=\frac{y_2}{y_1},\end{equation} on which $\sigma$ acts via \begin{gather} \label{zw-sigma}  \sigma(w)=\frac{b}{u\sigma(u)}w; \qquad \sigma(z)=z+w, \intertext{and $H$ acts via}  \label{zw-h}   \gamma(w)=\frac{\lambda_\gamma}{\alpha_\gamma}w; \qquad   \gamma(z)=\frac{\lambda_\gamma}{\alpha_\gamma}z+\frac{\beta_\gamma}{\alpha_\gamma}.\end{gather}
We observe that the $\sigma$-PV ring $$S=k[y_1, y_2+y_0, y_0, (y_1y_0)^{-1}]=k[y_1,w, z, (y_1w)^{-1}]$$ and the $\sigma\delta$-PV ring $$R=k\{y_1,y_2+y_0,y_0,(y_1y_0)^{-1}\}_\delta=k\{y_1,w, z, (y_1w)^{-1}\}_\delta.$$

Our computation of the $\sigma\delta$-Galois group $G$ for \eqref{difeq} in this section will be accomplished by studying the action of $G$ on $y_1$, $w$, and $z$. We begin by defining the \emph{unipotent radicals} \begin{equation} \label{uniprad-def} R_u(H)=H\cap\left\{\begin{pmatrix} 1& \beta \\ 0 & 1\end{pmatrix} \ \middle| \ \beta\in C\right\}\qquad\text{and}\qquad R_u(G)=G\cap\left\{\begin{pmatrix} 1& \beta \\ 0 & 1\end{pmatrix} \ \middle| \ \beta\in C\right\},\end{equation} and observe that $R_u(H)$ (resp., $R_u(G)$) is an algebraic (resp., differential algebraic) subgroup of $\mathbb{G}_a(C)$, the additive group of $C$. By \cite[Thm.~4.2(2)]{hendriks:1998}, $R_u(H)=\mathbb{G}_a(C)$ if and only if there exists exactly one solution $u\in k$ to \eqref{ric1}. We observe that $$R_u(G)=\{\gamma\in G \ | \ \gamma(y_i)=y_i \ \text{for} \ i=0,1\}.$$The reductive quotient $$G/R_u(G)\simeq\left\{\begin{pmatrix} \alpha_\gamma & 0 \\ 0 & \lambda_\gamma\end{pmatrix} \ \middle| \ \gamma\in G\right\}$$ is the $\sigma\delta$-Galois group corresponding to the matrix equation \begin{equation}\label{redeq}\sigma(Y)=\begin{pmatrix} u & 0 \\ 0 & v \end{pmatrix}Y,\end{equation} which we compute with Proposition~\ref{galdiag} and Remark~\ref{diagonalizable-rem}.

\begin{rem}[(Reduction to the connected case)] \label{connected-rem} Since $R_u(G)$ is connected, $G$ and $G/R_u(G)$ have the same number $t$ of connected components, and this number $t$ can be read off the description of the reductive quotient $G/R_u(G)$ provided by Proposition~\ref{galdiag}: in case~(i), $t=\mathrm{lcm}(m,n)$ (least common multiple); in case~(ii), $t=\mathrm{gcd}(m,n)$ (greatest common divisor) if $m$ and $n$ are both nonzero, otherwise $t$ coincides with whichever exponent is nonzero; in all other cases of Proposition~\ref{galdiag}, $G/R_u(G)$ is connected and $t=1$.

By Proposition~\ref{idemp}, the connected component of the identity $G^0$ is identified with $\mathrm{Gal}_{\sigma^t\delta}(R_0/k)$, where $R_0$ is a $\sigma^t\delta$-PV ring over the $\sigma^t\delta$-field $k$ for the system $$\sigma^t(Y)=B_tY, \quad \text{where} \ B_t=\sigma^{t-1}(B)\dots\sigma(B)B.$$ For each integer $n\geq 1$, we define the sequences \begin{align*}u_n &=\sigma^{n-1}(u)\dots\sigma(u)u=\sigma^{n-1}(u)u_{n-1} \\ v_n &=\sigma^{n-1}(v)\dots\sigma(v)v=\sigma^{n-1}(v)v_{n-1} \\ f_1 &=1; \ f_{n+1} = \sigma^n(u)f_n+ v_n \\ w_n &=\frac{f_n}{u_n}\cdot\frac{y_0}{y_1}, \end{align*} and observe that $u_1=u$, $v_1=v$, and $w_1=w$. We claim that $$B_n=\begin{pmatrix} u_n & f_n-u_n+v_n \\ 0 & v_n\end{pmatrix}.$$ We proceed by induction: $B=B_1$, and for each $n\geq 1$ we have $$B_{n+1}=\sigma^n(B)B_n =\begin{pmatrix}\sigma^n(u) & 1-\sigma^n(u)+\sigma^n(v) \\ 0 & \sigma^n(v)\end{pmatrix}\begin{pmatrix} u_n & f_n-u_n+v_n \\ 0 & v_n\end{pmatrix} =\begin{pmatrix} u_{n+1} & f_{n+1}-u_{n+1}+v_{n+1} \\ 0 & v_{n+1}\end{pmatrix}.$$ Since $$\sigma^n\left(\begin{pmatrix} y_1 & y_2+y_0 \\ 0 & y_0\end{pmatrix}\right)=B_n\begin{pmatrix} y_1 & y_2+y_0 \\ 0 & y_0\end{pmatrix},$$ we have that $\sigma^n(y_1)=u_ny_1$, $\sigma^n(y_0)=v_ny_0$, and $\sigma^n(y_2)=u_ny_2+f_ny_0$. Letting again $z=\tfrac{y_2}{y_1}$, we have that $$\sigma^n(z)=\frac{\sigma^n(y_2)}{\sigma^n(y_1)}=\frac{u_ny_2+f_ny_0}{u_ny_1}=z+w_n.$$ Hence, after replacing $\sigma$ with $\sigma^t$, $u$ with $u_t$, $v$ with $v_t$, and $w$ with $w_t$, we may assume in the rest of this section that $G$ is connected. Indeed, if we write $\tilde{x}=\frac{x}{t}$, $\tilde{\sigma}=\sigma^t$, and $\tilde{\delta}= t\delta$, then $k=C(\tilde{x})$, $\tilde{\sigma}(\tilde{x})=\tilde{x}+1$, and $\tilde{\delta}(\tilde{x})=1$, whence the replacement of $\sigma$ with $\sigma^t$ is immaterial for our purposes (n.b.: this observation already appears in \cite[Rem.~4.7]{hendriks:1998} and \cite[Rem.(1), p.~242]{hendriks-singer:1999}).
\end{rem}

In the following result, we compute the defining equations for the $\sigma\delta$-Galois group $G$ for \eqref{difeq} in a special case. Recall that $u\in\bar{\mathbb{Q}}(x)$ denotes the unique solution to the Riccati equation \eqref{ric1}, $H$ denotes the $\sigma$-Galois group for \eqref{difeq}, and $w$ is as in \eqref{zw-def}.

\begin{prop} \label{wrat} Suppose that $H$ is a connected subgroup of \begin{equation} \label{bigut}\mathbb{G}_m(C)\ltimes\mathbb{G}_a(C)=\left\{\begin{pmatrix} \alpha & \beta \\ 0 & \alpha\end{pmatrix} \ \middle| \ \alpha,\beta\in C, \ \alpha\neq 0\right\}\end{equation} with $R_u(H)=\mathbb{G}_a(C)$. Then $w\in\bar{\mathbb{Q}}(x)$, and $G$ is the subgroup of $\mathbb{G}_m(C)\ltimes\mathbb{G}_a(C)$ defined by one of the following conditions. \begin{enumerate}
\item[(i)] There exists $h\in\bar{\mathbb{Q}}(x)$ such that $u=\frac{\sigma(h)}{h}$ if and only if $\alpha=1$ and $R_u(G)=\mathbb{G}_a(C)$, i.e., $G=H=\mathbb{G}_a(C)$.
\item[(ii)] Case (i) does not hold and there exists $f\in \bar{\mathbb{Q}}(x)$ such that $\frac{\delta(u)}{u}=\sigma(f)-f$ if and only if $\delta(\alpha)=0$ and $R_u(G)=\mathbb{G}_a(C)$, i.e., $G=\mathbb{G}_m(C^\delta)\ltimes\mathbb{G}_a(C)$.
\item[(iii)] Cases (i) and (ii) do not hold and there exist $g\in \bar{\mathbb{Q}}(x)$ and a linear $\delta$-polynomial $\mathcal{L}\in \bar{\mathbb{Q}}\{Y\}_\delta$ such that $\mathcal{L}(\frac{\delta(u)}{u})-w=\sigma(g)-g$ if and only if $\beta=\alpha\mathcal{L}(\frac{\delta(\alpha)}{\alpha})$, i.e., $G\simeq\mathbb{G}_m(C)$, where the embedding $G\hookrightarrow\mathrm{GL}_2(C)$ associated to the choice of fundamental solution matrix \eqref{fundsol} is given by \begin{equation}\label{nondiagemb}\alpha\mapsto\begin{pmatrix}\alpha & \alpha\mathcal{L}\left(\tfrac{\delta(\alpha)}{\alpha}\right) \\ 0 & \alpha \end{pmatrix}.\end{equation}
\item[(iv)] If none of (i), (ii), or (iii) holds, then $G=H=\mathbb{G}_m(C)\ltimes\mathbb{G}_a(C)$.
\end{enumerate}
\end{prop}

\begin{proof} We recall the notation introduced at the beginning of this section: $v=\frac{b}{u}$, $\{y_1,y_2\}$ is a $C$-basis of solutions for \eqref{difeq} such that $\sigma(y_1)=uy_1$ and $\sigma(y_2)-uy_2=y_0$, where $y_0\neq 0$ satisfies $\sigma(y_0)=vy_0$. The embedding $H\hookrightarrow \mathrm{GL}_2(C):\gamma\mapsto T_\gamma$ is as in \eqref{h-emb}, and the action of $H$ on the solutions is given in \eqref{h-action}. The auxiliary elements $w$ and $z$ are defined as in \eqref{zw-def}; they are acted upon by $\sigma$ as in \eqref{zw-sigma} and by $H$ as in \eqref{zw-h}.

The relation $\gamma(w)=\frac{\lambda_\gamma}{\alpha_\gamma}w$ for each $\gamma\in H$ from \eqref{zw-h}, together with Theorem~\ref{correspondence}, imply that $w\in k$. Since $\sigma(w)=\frac{b}{u\sigma(u)}w$ from \eqref{zw-sigma} and $b,u\in\bar{\mathbb{Q}}(x)$, if $w\in k$ we may actually take $w\in \bar{\mathbb{Q}}(x)$ by \cite[Lem.~2.5]{hardouin:2008} (cf.~Remark~\ref{c-ok}).

Part (i) was proved in Proposition~\ref{galdiag}, except for the statement concerning the unipotent radical. In this case, we see that $R_u(H)$ is the $\sigma$-Galois group over $k$ for $\sigma(z)-z=w$. Since $R_u(H)=\mathbb{G}_a(C)$, there is no $g\in k$ such that $\sigma(g)-g=w$, for otherwise this $\sigma$-Galois group would be trivial. Therefore $z$ is $\delta$-transcendental over $k$ by \cite[Prop.~3.9(2)]{hardouin-singer:2008}, which implies that $R_u(G)=\mathbb{G}_a(C)$ by \cite[Prop.~6.26]{hardouin-singer:2008}. This proves (i).

Part (ii) was proved in Proposition~\ref{galdiag}, except for the statement concerning the unipotent radical. If (i) does not hold, then $H=\mathbb{G}_m(C)\ltimes\mathbb{G}_a(C)$ by \cite[Lem.~4.4]{hendriks:1998}. By Proposition~\ref{dense}, $G$ is Zariski-dense in $H$, and therefore $G$ is connected by \cite[Cor.~3.7]{minchenko-ovchinnikov-singer:2013b}. Proposition~\ref{idemp} says that the $\sigma\delta$-PV ring $$k\{y_1,y_0, (y_0y_1)^{-1}\}_\delta=k\{y_1,y_1^{-1}\}_\delta$$ is a domain, and its total ring of fractions $L$ is a field. If there exists $f\in\bar{\mathbb{Q}}(x)$ such that $\frac{\delta(u)}{u}=\sigma(f)-f$, then $L$ consists of $\delta$-algebraic elements over $k$ by \cite[Cor.~3.4(1)]{hardouin-singer:2008}. If $\gamma\in R_u(H)$, then $\gamma(z)=z+\beta_\gamma$ by \eqref{zw-h}. Since $\sigma(z)-z=w$ and the unipotent radical of $H$ is $\mathbb{G}_a(C)$, there is no $g\in k$ such that $w=\sigma(g)-g$, whence by \cite[Prop.~3.9(2)]{hardouin-singer:2008} $z$ must be $\delta$-transcendental over $k$. Therefore, $z$ is also $\delta$-transcendental over $L$, since $L$ consists of $\delta$-algebraic elements over $k$, which implies that $R_u(G)=\mathbb{G}_a(C)$. This proves (ii).

Since $$\gamma\left(\mathcal{L}\left(\frac{\delta(y_1)}{y_1}\right)-z\right)=\mathcal{L}\left(\frac{\delta(y_1)}{y_1}\right)-z + \left[\mathcal{L}\left(\frac{\delta(\alpha_\gamma)}{\alpha_\gamma}\right)-\frac{\beta_\gamma}{\alpha_\gamma}\right]$$ for any linear $\delta$-polynomial $\mathcal{L}\in C\{Y\}_\delta$ and $\gamma\in G$, Theorem~\ref{correspondence} implies that $\mathcal{L}(\frac{\delta(\alpha_\gamma)}{\alpha_\gamma})=\frac{\beta_\gamma}{\alpha_\gamma}$ for every $\gamma\in G$ if and only if \begin{equation}\label{wrat-2} \mathcal{L}\left(\frac{\delta(y_1)}{y_1}\right)-z=g\in k.\end{equation} Applying $\sigma-1$ to each side of \eqref{wrat-2}, we obtain \begin{equation} \label{wrat-2l}  \mathcal{L}\left(\frac{\delta(u)}{u}\right)-w=\sigma(g)-g.\end{equation} If $\mathcal{L}\in\bar{\mathbb{Q}}\{Y\}_\delta$, then the left-hand side of \eqref{wrat-2l} belongs to $\bar{\mathbb{Q}}(x)$, and therefore we may take $g\in\bar{\mathbb{Q}}(x)$ by \cite[Lem.~2.4]{hardouin:2008} (cf.~Remark~\ref{c-ok}). It is clear that if $\beta_\gamma=\alpha_\gamma\mathcal{L}(\frac{\delta(\alpha_\gamma)}{\alpha_\gamma})$ for every $\gamma\in G$, then $\beta_\gamma=0$ whenever $\alpha_\gamma=1$, so $R_u(G)=\{0\}$. This proves (iii).

Now suppose that neither (i) nor (ii) holds, i.e., there is no $f\in k$ such that $\frac{\delta(u)}{u}=\sigma(f)-f$. Then it follows from Corollary~\ref{diffprod} that $y_1$ is $\delta$-transcendental over $k$, and therefore $G/R_u(G)\simeq\mathbb{G}_m(C)$. In this case, we have that $G=\mathbb{G}_m(C)\ltimes\mathbb{G}_a(C)$ if and only if $R_u(G)=\mathbb{G}_a(C)$. Hence, to prove (iv) we have to show that if $R_u(G)\subsetneq\mathbb{G}_a(C)$ is a proper subgroup, then there exist $g\in \bar{\mathbb{Q}}(x)$ and a linear $\delta$-polynomial $\mathcal{L}\in \bar{\mathbb{Q}}\{Y\}_\delta$ as in (iii).

Since $w\in k\subset L$, it follows from Proposition~\ref{diffsum} that either $z$ is $\delta$-transcendental over $L$, or else there exist $h\in L$ and a nonzero linear $\delta$-polynomial $\mathcal{L}_0\in C\{Y\}_\delta$ such that $\mathcal{L}_0(w)=\sigma(h)-h$, which occurs if and only if $\mathcal{L}_0(z)\in L$. On the other hand, since $\gamma(\mathcal{L}_0(z))=\mathcal{L}_0(z)+\mathcal{L}_0(\beta_\gamma)$ for each $\gamma\in R_u(G)=\mathrm{Gal}_{\sigma\delta}(R/L)$, Theorem~\ref{correspondence} implies that $z$ is $\delta$-transcendental over $L$ if and only if $R_u(G)=\mathbb{G}_a(C)$. We will show that if \begin{equation}\label{wrat-3}\mathcal{L}_0(w)=\sigma(h)-h\end{equation} for some $h\in L$ and nonzero linear $\delta$-polynomial $\mathcal{L}_0\in C\{Y\}_\delta$, then there exist $g\in \bar{\mathbb{Q}}(x)$ and $\mathcal{L}\in \bar{\mathbb{Q}}\{Y\}_\delta$ as in (iii).

Since $\mathcal{L}_0(w)\in k$, it follows from \eqref{wrat-3} that the $k$-vector space $k\langle\sigma^i(h)\rangle_{i\in\mathbb{Z}}$ is finite dimensional, and therefore Lemma~\ref{sfin} implies that $h$ belongs to the $\sigma\delta$-PV ring $k\{ y_1,y_1^{-1}\}_\delta$. We claim that in fact $$h\in P=k\left\{\frac{\delta(y_1)}{y_1}\right\}_\delta.$$ To see this, let $P(n)=y_1^n\cdot P$ for $n\in\mathbb{Z}$, and observe the decomposition $$k\{y_1,y_1^{-1}\}_\delta=P[y_1,y_1^{-1}]=\bigoplus_{n\in\mathbb{Z}}P(n)$$ into $\sigma$-stable $k$-vector spaces (the sum is direct because $y_1$ is $\delta$-transcendental over $k$, and therefore algebraically transcendental over $P$). Since $\mathcal{L}_0(w)\in k\subset P(0)$, we may assume that $\mathcal{L}_0(z)=h\in P(0)$ also, which implies that $z$ and $\frac{\delta(y_1)}{y_1}$ are $\delta$-dependent over $k$. By Proposition~\ref{diffsum}, there must exist linear $\delta$-polynomials $\mathcal{L}_1,\mathcal{L}_2\in C\{Y\}_\delta$ and $\tilde{g}\in k$ such that \begin{equation}\label{wrat-4}\mathcal{L}_1\left(\frac{\delta(u)}{u}\right)-\mathcal{L}_2(w)=\sigma(\tilde{g})-\tilde{g}.\end{equation}

We will construct a linear $\delta$-polynomial $$\mathcal{L}\in \bar{\mathbb{Q}}\{Y\}_\delta \quad\text{such that}\quad\mathcal{L}\left(\frac{\delta(u)}{u}\right)-w=\sigma(g)-g \quad\text{for some}\quad g\in \bar{\mathbb{Q}}(x).$$ Let $\mathrm{ord}(\mathcal{L}_i)=m_i$ and $\mathcal{L}_i=\sum_{j=0}^{m_i}c_{i,j}\delta^j(Y)$ for $i=1,2$. By Proposition~\ref{indsum}, the existence of $\tilde{g}\in k$ as in \eqref{wrat-4} is equivalent to \begin{equation}\label{wrat-5} 0=\mathrm{dres}_x\left(\mathcal{L}_1\left(\frac{\delta(u)}{u}\right)-\mathcal{L}_2(w), [d],n\right)\end{equation} for every $\mathbb{Z}$-orbit $[d]$ and every $n\in\mathbb{N}$. Let $r\in\mathbb{N}$ be the largest order such that $\mathrm{dres}_x(w,[d],r)\neq 0$ for some orbit $[d]$. Then it follows from \eqref{wrat-5} that, for each orbit $[d]$, the discrete residues \begin{gather*}c_{1,m_2+r-1}(-1)^{m_2+r-1}(m_2+r-1)!\mathrm{dres}_x\left(\frac{\delta(u)}{u},[d],1\right)=\mathrm{dres}_x\left(\mathcal{L}_1\left(\frac{\delta(u)}{u}\right),[d],m_2+r\right) \intertext{and} c_{2,m_2}(-1)^{m_2}\frac{(m_2+r)!}{(r-1)!}\mathrm{dres}_x(w,[d],r)=\mathrm{dres}_x(\mathcal{L}_2(w),[d],m_2+r)\end{gather*} are equal.

We define the leading coefficient $c_{r-1}$ of $\mathcal{L}$ by $$\mathcal{L}=\frac{c_{1,m_2+r-1}}{(m_2+r)c_{2,m_2}}\delta^{r-1}(Y)+\dots$$ and observe that \begin{equation}\label{wrat-6}\mathrm{dres}_x\left(c_{r-1}\delta^{r-1}\left(\frac{\delta(u)}{u}\right)-w,[d],r\right)=0\end{equation} for every $[d]$. Since $u,w\in\bar{\mathbb{Q}}(x)$ and $\mathrm{dres}_x(h,[d],n)$ is $C$-linear in $h$, it follows that the leading coefficient $c_{r-1}\in\bar{\mathbb{Q}}$.

We continue by taking the next highest $r'\leq r-1$ such that $\mathrm{dres}_x(w,[d],r')\neq 0$ for some $[d]$, and proceed as above to find the coefficient $c_{r'-1}\in \bar{\mathbb{Q}}$ of $\mathcal{L}$ such that \eqref{wrat-6} holds with $r'$ in place of $r$. Eventually we will have constructed a linear $\delta$-polynomial $\mathcal{L}\in \bar{\mathbb{Q}}\{Y\}_\delta$ such that $$\mathrm{dres}_x\left(\mathcal{L}\left(\frac{\delta(u)}{u}\right)-w,[d],n\right)=0$$ for each $[d]$ and $n\in\mathbb{N}$, which by Proposition~\ref{indsum} implies that $\mathcal{L}(\frac{\delta(u)}{u})-w=\sigma(g)-g$ for some $g\in k$. By \cite[Lem.~2.4]{hardouin:2008} (cf.~Remark~\ref{c-ok}), we may take $g\in \bar{\mathbb{Q}}(x)$, so we are indeed in case (iii), as we wanted to show.\end{proof}

\begin{rem} The situation described in Proposition~\ref{wrat}(iii) relates to a phenomenon for linear differential algebraic groups that has no analogue in the theory of linear algebraic groups. Namely, the existence of the logarithmic derivative map $\mathbb{G}_m(C)\twoheadrightarrow\mathbb{G}_a(C):\alpha\mapsto\frac{\delta(\alpha)}{\alpha}.$ This is what allows for the embeddings $\mathbb{G}_m(C)\hookrightarrow\mathrm{GL}_2(C)$ described in \eqref{nondiagemb} and the anomalous situation of Proposition~\ref{wrat}(iii) where $R_u(G)=\{0\}$ despite the fact that $R_u(H)=\mathbb{G}_a(C)$. We will see in Proposition~\ref{bigut2} that, if the connected component $H^0$ is not conjugate to a subgroup of $\mathbb{G}_m(C)\ltimes\mathbb{G}_a(C)$ as in \eqref{bigut}, then this phenomenon does not occur and we always have $R_u(G)=R_u(H)$.
\end{rem}

The following observation, made near the beginning of the proof of Proposition~\ref{wrat}, will be useful in the proof of Theorem~\ref{rad}.

\begin{cor} \label{wrat-cor} $H$ is a subgroup of $\mathbb{G}_m(C)\ltimes\mathbb{G}_a(C)$ as in \eqref{bigut} if and only if $w\in k$, where $w$ is defined as in \eqref{zw-def}. In this case, we have that $R_u(G)$ is either $\{0\}$ or $\mathbb{G}_a(C)$.\end{cor}

The following theorem is the main theoretical result of this section. The result holds true assuming only that $G$ is the $\sigma\delta$-Galois group for \eqref{difeq}, i.e., there is no hypothesis concerning the existence of solutions $u\in \bar{\mathbb{Q}}(x)$ to the Riccati equations \eqref{ric1} or \eqref{ric2}. However, since $R_u(G)\subseteq R_u(H)$ in general, the only non-trivial case of the theorem occurs when we assume that $R_u(H)\neq \{0\}$. By the results of \cite{hendriks:1998} summarized in \S\ref{hendriks-sec}, this is equivalent to assuming that there exists precisely one solution $u\in\bar{\mathbb{Q}}(x)$ to \eqref{ric1}, so we will keep this assumption throughout the course of the proof of Theorem~\ref{rad}.

\begin{thm}\label{rad} The unipotent radical $R_u(G)$ of the $\sigma\delta$-Galois group $G$ of \eqref{mateq} is algebraic. In other words, $R_u(G)$ is either $\{0\}$ or $\mathbb{G}_a(C)$.
\end{thm}

\begin{proof} By Proposition~\ref{classification}, either $R_u(G)=\mathbb{G}_a(C)$, or else \begin{equation}\label{radop}R_u(G)=\left\{\begin{pmatrix} 1 & \beta \\ 0 & 1\end{pmatrix} \ \middle| \ \beta\in C, \ \mathcal{L}(\beta)=0\right\}.\end{equation} for some nonzero, monic linear $\delta$-polynomial $\mathcal{L}\in C\{Y\}_\delta$. We recall the notation introduced at the beginning of this section: $u\in\bar{\mathbb{Q}}(x)$ is the unique solution in $k$ to the Riccati equation~\eqref{ric1}, $v=\frac{b}{u}$, $\{y_1,y_2\}$ is a $C$-basis of solutions for \eqref{difeq} such that $\sigma(y_1)=uy_1$ and $\sigma(y_2)-uy_2=y_0$, where $y_0\neq 0$ satisfies $\sigma(y_0)=vy_0$. The embedding $G\hookrightarrow \mathrm{GL}_2(C):\gamma\mapsto T_\gamma$ is as in \eqref{h-emb}, and the action of $G$ on the solutions is given in \eqref{h-action}. The auxiliary elements $w$ and $z$ are defined as in \eqref{zw-def}; they are acted upon by $\sigma$ as in \eqref{zw-sigma} and by $G$ as in \eqref{zw-h}.

\begin{lem}\label{wrels} The following are equivalent:
\begin{enumerate}
\item[(i)] $w$ is $\delta$-algebraic over $k$;
\item[(ii)] $\sigma(w)=c\frac{\sigma(f)}{f}w$ for some nonzero $c\in C^\delta$ and $f\in k$;
\item[(iii)] $\frac{\delta(w)}{w}\in k$;
\item[(iv)] $\delta(\alpha_\gamma\lambda_\gamma^{-1})=0$ for every $\gamma\in G$.
\end{enumerate}
\end{lem}

\begin{proof} By \cite[Cor.~3.4(1)]{hardouin-singer:2008}, (i) and (ii) are equivalent. It is clear that (iii) $\Rightarrow$ (i). If $w$ is as in (ii), then $$\sigma\left(\frac{\delta(w)}{w}\right)=\frac{\delta(\sigma(w))}{\sigma(w)}=\frac{\delta(w)}{w}+\sigma\left(\frac{\delta(f)}{f}\right)-\frac{\delta(f)}{f},$$ which implies that $\frac{\delta(w)}{w}-\frac{\delta(f)}{f}=d$ for some $d\in C$. Hence, (ii) $\Rightarrow$ (iii). It follows from \eqref{zw-h} that $$\gamma\left(\frac{\delta(w)}{w}\right)=\frac{\delta(w)}{w}-\frac{\delta(\alpha_\gamma\lambda_\gamma^{-1})}{\alpha_\gamma\lambda_\gamma^{-1}},$$ and therefore Theorem~\ref{correspondence} implies that (iii) and (iv) are equivalent. This concludes the proof of Lemma~\ref{wrels}.\end{proof}

The following result establishes Theorem~\ref{rad} under the supplementary assumption that $w$ is $\delta$-transcendental over $k$.

\begin{lem}\label{wdtran} Suppose that $w$ is $\delta$-transcendental over $k$. Then $R_u(G)$ is either $\{0\}$ or $\mathbb{G}_a(C)$.
\end{lem}

\begin{proof} Since $R_u(G)$ is normal in $G$, this implies that (cf. \cite[Lem.~3.6]{hardouin-singer:2008})$$T_\gamma\begin{pmatrix} 1 & \beta \\ 0 &1\end{pmatrix}T_\gamma^{-1}=\begin{pmatrix}\alpha_\gamma & \tilde{\beta}_\gamma \\ 0 & \lambda_\gamma\end{pmatrix}\begin{pmatrix} 1 & \beta \\ 0 &1\end{pmatrix}\begin{pmatrix}\alpha_\gamma^{-1} & -\alpha_\gamma^{-1}\lambda_\gamma^{-1}\tilde{\beta}_\gamma \\ 0 & \lambda_\gamma^{-1}\end{pmatrix}=\begin{pmatrix} 1 & \alpha_\gamma\lambda_\gamma^{-1}\beta \\ 0 & 1\end{pmatrix}\in R_u(G).$$ If $\mathcal{L}$ is as in \eqref{radop}, then $\mathcal{L}(\beta)=0\Rightarrow\mathcal{L}(\alpha_\gamma\lambda_\gamma^{-1}\beta)=0$ for each $\gamma\in G$ and $\left(\begin{smallmatrix} 1 & \beta \\ 0 & 1\end{smallmatrix}\right)\in R_u(G)$. By \cite[Lem.~3.7]{hardouin-singer:2008}, this implies that if $\mathrm{ord}(\mathcal{L})\neq 0$, then $\delta(\alpha_\gamma\lambda_\gamma^{-1})=0$ for every $\gamma\in G$, contradicting the hypothesis that $w$ is $\delta$-transcendental by Lemma~\ref{wrels}. This concludes the proof of Lemma~\ref{wdtran}. 
\end{proof}

It remains to prove Theorem~\ref{rad} when $w$ is $\delta$-algebraic over $k$. The case where $w\in k$ is treated in Corollary~\ref{wrat-cor}, so we may assume from now on that $w\notin k$. We give different arguments depending on whether $y_1$ is $\delta$-algebraic over $k$ (Lemma~\ref{wdalg1}) or $\delta$-transcendental over $k$ (Lemma~\ref{wdalg2}). We begin with a preliminary result.

\begin{lem}\label{zdalg} If $w$ is $\delta$-algebraic over $k$, then $z$ is $\delta$-transcendental over $k$. 
\end{lem}

\begin{proof} We proceed by contradiction: assuming that $w$ and $z$ are both $\delta$-algebraic over $k$, we will show that $R_u(H)=\{0\}$, contradicting our assumption that there exists exactly one solution $u\in\bar{\mathbb{Q}}(x)$ to the Riccati equation \eqref{ric1}. By Lemma~\ref{wrels}, $\sigma(w)=c\frac{\sigma(f)}{f}w$ for some nonzero $c\in C^\delta$ and $f\in k$. Since $$\sigma\left(\frac{fz}{w}\right)=\frac{\sigma(f)(z+w)}{c\tfrac{\sigma(f)}{f}w}=c^{-1}\left(\frac{fz}{w}\right)+c^{-1}f,$$ the element $w^{-1}fz$ is $\delta$-algebraic over $k$ and satisfies \begin{equation}\label{gosper2} \sigma(Y)-c^{-1}Y=c^{-1}f.\end{equation} Hence, \cite[Prop.~3.9(2)]{hardouin-singer:2008} implies that there exists an element $h\in k$ such that $\sigma(h)-c^{-1}h=c^{-1}f$, and therefore $$\sigma\left(\frac{fz}{w}-h\right)=c^{-1}\left(\frac{fz}{w}-h\right).$$ Hence, $w^{-1}fz-h=dw^{-1}f$ for some $d\in C$, whence $z=hwf^{-1}+df^{-1}$ and $R_u(H)=\{0\}$. This concludes the proof of Lemma~\ref{zdalg}.
\end{proof}

If $\mathcal{L}\in C\{Y\}_\delta$ is a linear $\delta$-polynomial as in \eqref{radop}, then for every $\gamma\in R_u(G)$ we have that $$\gamma\bigl(\mathcal{L}(z)\bigr)=\mathcal{L}(z)+\mathcal{L}(\beta_\gamma)=\mathcal{L}(z),$$ and therefore Theorem~\ref{correspondence} implies that $\mathcal{L}(z)\in L$, the total ring of fractions of the $\sigma\delta$-PV ring $k\{y_1,y_0, (y_1y_0)^{-1}\}_\delta$ for \eqref{redeq}. By Remark~\ref{connected-rem}, we may assume that $G$ is connected, in which case $L$ is a field by Proposition~\ref{idemp}.

\begin{lem}\label{wdalg1} Suppose that $G$ is connected and $w$ and $y_1$ are both $\delta$-algebraic over $k$. Then $R_u(G)=\mathbb{G}_a(C)$.
\end{lem}

\begin{proof}We proceed by contradiction: assuming that $w$ and $y_1$ are both $\delta$-algebraic over $k$ and that $R_u(G)\neq\mathbb{G}_a(C)$ is proper, we will show that $z$ is also $\delta$-algebraic over $k$, contradicting Lemma~\ref{zdalg}. Since $y_0=uwy_1$ is also $\delta$-algebraic over $k$, the ring of fractions $L$ of the $\sigma\delta$-PV ring $k\{y_1,y_0,(y_1y_0)^{-1}\}_\delta$ consists of $\delta$-algebraic elements over $k$. Since $G$ is connected, $L$ is a field. If $R_u(G)\neq\mathbb{G}_a(C)$ is as in \eqref{radop}, then $\mathcal{L}(z)\in L$ by Theorem~\ref{correspondence}. Hence, $\mathcal{L}(z)$ is $\delta$-algebraic over $k$, and therefore $z$ is also $\delta$-algebraic over $k$, concluding the proof of Lemma~\ref{wdalg1}.
\end{proof}

\begin{lem}\label{wdalg2} Suppose that $G$ is connected, $w\notin k$ is $\delta$-algebraic over $k$, and $y_1$ is $\delta$-transcendental over $k$. Then $R_u(G)=\mathbb{G}_a(C)$.
\end{lem}

\begin{proof} We proceed by contradiction: assuming that $w\notin k$ is $\delta$-algebraic over $k$, $y_1$ is $\delta$-transcendental over $k$, and $R_u(G)\neq\mathbb{G}_a(C)$ is proper, we will show that $z$ is also $\delta$-algebraic over $k$, contradicting Lemma~\ref{zdalg}. By Lemma~\ref{wrels}, $\sigma(w)=c\frac{\sigma(f)}{f}w$ for some nonzero $c\in C^\delta$ and $f\in k$, and $\frac{\delta(w)}{w}=g_1\in k$. We remark that $w\notin k$ implies that $c\neq 1$, and that since $G$ is connected, $c$ is not a root of unity and $w$ is algebraically transcendental over $k$. The following induction argument shows that $\frac{\delta^r(w)}{w}=g_r\in k$ for every $r\in\mathbb{N}$: if $\frac{\delta^{r-1}(w)}{w}=g_{r-1}\in k$, then $$g_r=\frac{\delta^r(w)}{w}=\delta\left(\frac{\delta^{r-1}(w)}{w}\right)+\frac{\delta^{r-1}(w)\delta(w)}{w^2}=\delta(g_{r-1})+g_{r-1}g_1\in k.$$ If $\mathcal{L}\in C\{Y\}_\delta$ is a linear $\delta$-polynomial as in \eqref{radop}, there exists a $g_\mathcal{L}\in k$ such that $\mathcal{L}(w)=g_\mathcal{L}w.$ It follows from the relations $$\sigma(\mathcal{L}(z))-\mathcal{L}(z)=\mathcal{L}(w)=g_\mathcal{L}w \qquad\text{and}\qquad \sigma(g_\mathcal{L}w)=c\frac{\sigma(g_\mathcal{L}f)}{g_\mathcal{L}f}w$$ that the $k$-vector space $k\langle\sigma^i(\mathcal{L}(z))\rangle_{i\in\mathbb{Z}}$ is finite-dimensional over $k$. Hence, Lemma~\ref{sfin} implies that $\mathcal{L}(z)\in L$ actually belongs to the $\sigma\delta$-PV ring $k\{y_1,y_0,(y_1y_0)^{-1}\}_\delta$ for \eqref{redeq}.

Since $\frac{\delta(y_0)}{y_0}=g_1+\frac{\delta(y_1)}{y_1}+\frac{\delta(u)}{u}$, we see that the $\sigma\delta$-PV ring $$k\{y_1,y_0,(y_1y_0)^{-1}\}_\delta=k\left[y_1,y_1^{-1},w,w^{-1},\delta^i\left(\frac{\delta(y_1)}{y_1}\right)\right]_{i\in\mathbb{Z}_{\geq 0}}.$$ We define $$P=k\left[\delta^i\left(\frac{\delta(y_1)}{y_1}\right)\right]_{i\in\mathbb{Z}_{\geq 0}},$$ and observe that we have a decomposition into $\sigma$-invariant $k$-vector spaces $$k\{ y_1,w,(y_1w)^{-1}\}_\delta\simeq\bigoplus_{m_1,m_0}P(m_1,m_0),\quad \text{where} \quad P(m_1,m_0)=y_1^{m_1}w^{m_0}\cdot P,$$ and the sum is direct because: $w$ is algebraically transcendental over $k$; $y_1$ is $\delta$-transcendental over $k$; and therefore $y_1$ is also $\delta$-transcendental over $k(w)$, since $w$ is $\delta$-algebraic over $k$; which implies that $y_1$ is algebraically transcendental over $P[w,w^{-1}]$. Hence, the relation $$\sigma(\mathcal{L}(z))-\mathcal{L}(z)=g_\mathcal{L}w\in P(0,1)$$ implies that $\mathcal{L}(z)-d\in P(0,1)$ for some $d\in C=P^\sigma$.

We order the monomials in $P$ lexicographically, as follows: first compare the orders of the highest derivatives of $\frac{\delta(y_1)}{y_1}$ appearing in each monomial. If these are the same $r$, then compare the algebraic exponents of $\delta^r(\frac{\delta(y_1)}{y_1})$. If these are the same, then compare the exponents of $\delta^{r-1}(\frac{\delta(y_1)}{y_1})$, and so on. Then $$\mathcal{L}(z)=hw\left(\delta^r\left(\frac{\delta(y_1)}{y_1}\right)\right)^{n_r}\dots\left(\frac{\delta(y_1)}{y_1}\right)^{n_0} +\ (\text{lower-order terms})$$ for some $h\in k$, and applying $\sigma-1$ on both sides we obtain that either $n_i=0$ for every $i$, or else the leading term $$\left(\sigma(h)c\frac{\sigma(f)}{f}w-hw\right)\left(\delta^r\left(\frac{\delta(y_1)}{y_1}\right)\right)^{n_r}\dots\left(\frac{\delta(y_1)}{y_1}\right)^{n_0}=0,$$ which would imply that $\sigma(hf)=c^{-1}hf$. But by \cite[Cor.~2.3]{vanderput-singer:1997}, the $\sigma$-Galois group of $\sigma(y)=c^{-1}y$ over $k$ is nontrivial whenever $c\neq 1$, in which case there is no nonzero element of $k$ that satisfies $\sigma(y)=c^{-1}y$. Hence, $\mathcal{L}(z)=hw$ is $\delta$-algebraic over $k$, which implies that $z$ is also $\delta$-algebraic over $k$, concluding the proof of Lemma~\ref{wdalg2}.
\end{proof}

To summarize: since $R_u(G)$ is connected, we may assume that $G$ is connected by Remark~\ref{connected-rem}. When $w$ is $\delta$-transcendental over $k$, Theorem~\ref{rad} follows from Lemma~\ref{wdtran}. When $w\notin k$ is $\delta$-algebraic over $k$, Theorem~\ref{rad} follows from Lemma~\ref{wdalg1} when $y_1$ is also $\delta$-algebraic over $k$, and from Lemma~\ref{wdalg2} when $y_1$ is $\delta$-transcendental over $k$. Finally, when $w\in k$, Theorem~\ref{rad} follows from Corollary~\ref{wrat-cor}, concluding the proof of Theorem~\ref{rad}.
\end{proof}

In the following result, we apply Theorem~\ref{rad} to conclude the computation of the $\sigma\delta$-Galois group $G$ for \eqref{difeq} begun in Proposition~\ref{galdiag}, Remark~\ref{diagonalizable-rem}, Remark~\ref{connected-rem}, and Proposition~\ref{wrat}. We recall that $H$ denotes the $\sigma$-Galois group for \eqref{difeq}. The unipotent radicals $R_u(H)$ and $R_u(G)$ are defined as in \eqref{uniprad-def}. By Remark~\ref{connected-rem}, we may assume that $H$ and $G$ are connected.

\begin{prop} \label{bigut2} Suppose that $H$ is connected, and not conjugate to a subgroup of $$\left\{\begin{pmatrix} \alpha & \beta \\ 0 & \alpha\end{pmatrix} \ \middle| \ \alpha,\beta\in C, \ \alpha\neq 0\right\}.$$ Then $R_u(G)=R_u(H)$.
\end{prop}

\begin{proof} We keep the notation introduced at the beginning of this section: $u\in\bar{\mathbb{Q}}(x)$ is the unique solution in $k$ to the Riccati equation~\eqref{ric1}, $v=\frac{b}{u}$, $\{y_1,y_2\}$ is a $C$-basis of solutions for \eqref{difeq} such that $\sigma(y_1)=uy_1$ and $\sigma(y_2)-uy_2=y_0$, where $y_0\neq 0$ satisfies $\sigma(y_0)=vy_0$. The action of $G$ on the solutions is given in \eqref{h-action}. The auxiliary elements $w$ and $z$ are defined as in \eqref{zw-def}; they are acted upon by $\sigma$ as in \eqref{zw-sigma} and by $G$ as in \eqref{zw-h}.

Since $R_u(G)\subseteq\mathbb{G}_a(C)$ is algebraic by Theorem~\ref{rad}, we only have to show that $R_u(G)\neq\{0\}$. By Corollary~\ref{wrat-cor}, $H$ not being a subgroup of $\mathbb{G}_m(C)\ltimes\mathbb{G}_a(C)$ as in \eqref{bigut} is equivalent to assuming that $w\notin k$. If $w$ is $\delta$-algebraic over $k$, it follows from Lemma~\ref{wdalg1} and Lemma~\ref{wdalg2} that $R_u(G)=\mathbb{G}_a(C)$.

We proceed by contradiction: assuming that $w$ is $\delta$-transcendental over $k$ and $R_u(G)=\{0\}$, we will show that $R_u(H)=\{0\}$. Let $M$ be the total ring of fractions of $R$, and let $L$ be the total ring of fractions of the $\sigma\delta$-PV ring $k\{y_0,y_1,(y_0y_1)^{-1}\}_\delta$ for the system \eqref{redeq}. Since $G$ is connected, Proposition~\ref{idemp} implies that $M$ and $L$ are fields. Since $\{0\}=R_u(G)=\mathrm{Gal}_{\sigma\delta}(M/L)$, Theorem~\ref{correspondence} implies that $M= L$, and in particular $z\in L$. Consider the subfield $T\subset L$ obtained by taking the field of fractions of the $\sigma\delta$-PV ring $P=k\{\frac{\delta(y_0)}{y_0},\frac{\delta(y_1)}{y_1}\}_\delta$, and note that $L= T(y_0,y_1)$.

\begin{claim}\label{zspec}
There exist $g\in P$ and $d\in C$ such that $z=gw+d$.
\end{claim}
\begin{proof}
Recall from \eqref{zw-sigma} that $\sigma(w)=\tfrac{b}{u\sigma(u)}w$ and $\sigma(z)=z+w$. This implies that $\frac{z}{w}$ satisfies \begin{equation} \label{zow} \sigma\left(\frac{z}{w}\right)-\frac{u\sigma(u)}{b}\left(\frac{z}{w}\right)=\frac{u\sigma(u)}{b}.\end{equation} Since $G$ is connected, so is the reductive quotient $G/R_u(G)$, which we computed in Proposition~\ref{galdiag}. Since $w$ is $\delta$-algebraic over $k$, so is $\frac{y_0}{y_1}$, so case~(i) of Proposition~\ref{galdiag} does not hold under our present assumptions. But it is still possible for $G/R_u(G)$ to be as in case~(ii) of Proposition~\ref{galdiag}, provided that either the integers $m$ and $n$ are both nonzero and relatively prime, or else $\{m,n\}=\{0,1\}$ (cf.~Remark~\ref{connected-rem}).

We will first prove the claim under the assumption that case~(ii) of Proposition~\ref{galdiag} does not hold. Concretely, this means that $y_1$ and $y_0$ are algebraically independent over $T$. Since the coefficient $\frac{u\sigma(u)}{b}\in k\subset T$, a twofold application of \cite[Lem.~6.5]{hardouin-singer:2008} shows that there exists $g\in T$ such that $\sigma(g)-\frac{u\sigma(u)}{b}g=\frac{u\sigma(u)}{b}$. This implies that the $k$-vector space $k\langle \sigma^i(g)\rangle_{i\in\mathbb{Z}}$ is finite dimensional, and therefore $g\in P$ by Lemma~\ref{sfin}. It follows from $\sigma(\frac{z}{w}-g)=\frac{u\sigma(u)}{b}(\frac{z}{w}-g)$ and Theorem~\ref{nnc} that $\frac{z}{w}-g=dw^{-1}$ for some $d\in C$, concluding the proof of the claim in this case.

It remains to prove the claim when the reductive quotient $G/R_u(G)$ is as in case~(ii) of Proposition~\ref{galdiag} (where either $m$ and $n$ are both nonzero and relatively prime, or else $\{m,n\}=\{0,1\}$). Let us assume without loss of generality that $m\neq 0$. Since $\gamma(y_1^my_0^n)=\alpha_\gamma^m\lambda_\gamma^n(y_1^my_0^n)=y_1^my_0^n$, Theorem~\ref{correspondence} implies that $f=y_1^my_0^n\in k$. Since $w$ is $\delta$-transcendental over $k$, so is $\frac{y_0}{y_1}$, and it follows that $y_0$ is $\delta$-transcendental over $k$, and therefore algebraically transcendental over $T$. We may assume that $m>0$ is the smallest positive integer such that $y_1^m\in T(y_0)$. If $m>1$, then $\{1,y_1,\dots,y_1^{m-1}\}$ is a basis for $L=T(y_0,y_1)$ as a vector space over $T(y_0)$, so we may write $\frac{z}{w}\in L= T(y_0)(y_1)$ uniquely as a polynomial expression $\frac{z}{w}=q_0+q_1y_1+\dots+q_{m-1}y_1^{m-1}$, where each $q_i\in T(y_0)$. If we substitute this expression for $\frac{z}{w}$ in \eqref{zow}, we obtain $$\sigma\left(\sum_{i=0}^{m-1} q_iy_1^i\right)-\frac{u\sigma(u)}{b}\sum_{i=0}^{m-1} q_iy_1^i=\sum_{i=0}^{m-1} \left(\sigma(q_i)u^i-\frac{u\sigma(u)}{b}q_i\right)y_1^i=\frac{u\sigma(u)}{b},$$ which implies in particular that $q_0\in T(y_0)$ satisfies $\sigma(q_0)-\frac{u\sigma(u)}{b}q_0=\frac{u\sigma(u)}{b}$. Note that the existence of such an element $q_0$ is immediate when $m=1$. Since $y_0$ is algebraically transcendental over $T$, \cite[Lem.~6.5]{hardouin-singer:2008} implies that there exists $g\in T$ such that $\sigma(g)-\frac{u\sigma(u)}{b}g=\frac{u\sigma(u)}{b}$. This implies that the $k$-vector space $k\langle \sigma^i(g)\rangle_{i\in\mathbb{Z}}$ is finite dimensional, and therefore $g\in P$ by Lemma~\ref{sfin}. It follows from $\sigma(\frac{z}{w}-g)=\frac{u\sigma(u)}{b}(\frac{z}{w}-g)$ and Theorem~\ref{correspondence} that $\frac{z}{w}-g=dw^{-1}$ for some $d\in C$, which concludes the proof of the claim.\end{proof}

We will show that the element $g\in P$ described in Claim~\ref{zspec} actually belongs to $k$, which will imply that $z=gw+d \in k(w)$ and therefore $R_u(H)=0$, a contradiction. We order the monomials in $P$ lexicographically, as follows: if $y_0$ is $\delta$-transcendental over $k$, compare the orders of the highest derivatives of $\frac{\delta(y_0)}{y_0}$ appearing in each monomial. If these are the same $r_0$, then compare the algebraic exponents of $\delta^{r_0}(\frac{\delta(y_0)}{y_0})$ appearing in each monomial. If these are the same $n_{0,r_0}$, then compare the algebraic exponents of $\delta^{r_0-1}(\frac{\delta(y_0)}{y_0})$, and so on. If $y_0\in T$ or if $\delta^n(\frac{\delta(y_0)}{y_0})$ does not occur in the monomials for any $n\in\mathbb{N}$, then compare the highest orders of the derivatives of $\frac{\delta(y_1)}{y_1}$, and if these are the same $r_1$, then compare the algebraic exponents of $\delta^{r_1}(\frac{\delta(y_1)}{y_1})$, and so on, as with $y_0$. Then \begin{equation}\label{zlast} g=h\left(\delta^{r_0}\left(\frac{\delta(y_0)}{y_0}\right)\right)^{n_{0,r_0}}\dots\left(\frac{\delta(y_0)}{y_0}\right)^{n_{0,0}} \left(\delta^{r_1}\left(\frac{\delta(y_1)}{y_1}\right)\right)^{n_{1,r_1}}\dots\left(\frac{\delta(y_1)}{y_1}\right)^{n_{1,0}} + \ \text{(lower-order terms)},\end{equation} for some $h\in k$. If we substitute this expression for $g$ in $$\sigma(z)-z=\sigma(gw+d)-(gw+d)=\sigma(gw)-gw=w,$$ we see that either $n_{i,j_i}=0$ for every $0\leq j_i\leq r_i$, $i=1,2$, in which case $g\in k$, or else the leading coefficient $\sigma(h)\frac{b}{u\sigma(u)}w-hw$ of $\sigma(gw)-gw$ must be zero. But then $\sigma(h)=\frac{u\sigma(u)}{b}h$ implies that $h=cw^{-1}$ for some $c\in C$, contradicting the assumption that $w\notin k$. Therefore, $g\in k$ and $R_u(H)=\{0\}$. This contradiction concludes the proof of Proposition~\ref{bigut2}.
\end{proof}

\begin{rem} To compute the difference-differential Galois group $G$ for \eqref{difeq} when there is only one solution $u\in\bar{\mathbb{Q}}(x)$ to the Riccati equation \eqref{ric1}, we proceed as follows. We recall that $R_u(H)=\mathbb{G}_a(C)$ in this case. We write $v=\frac{b}{u}$, and compute the reductive quotient $G/R_u(G)$ as in Proposition~\ref{galdiag} and Remark~\ref{diagonalizable-rem}. We then compute the number of connected components $t$ of $G/R_u(G)$ as in Remark~\ref{connected-rem}, and proceed to compute $G^0$, the connected component of the identity in $G$.

Keeping the notation introduced in Remark~\ref{connected-rem}, we now proceed to check whether the hypotheses of Proposition~\ref{wrat} are satisfied (that is, whether $H^0$ is a subgroup of $\mathbb{G}_m(C)\ltimes\mathbb{G}_a(C)$ as in \eqref{bigut}), which by Corollary~\ref{wrat-cor} is the same as deciding whether $w_t\in\bar{\mathbb{Q}}(x)$. Note that $w_t$ satisfies \begin{equation}\label{wrat-check}\sigma^t(y)=\frac{\sigma^t(f_t)v_t}{\sigma^t(u_t)}y,\end{equation} and we may use the results of \cite{abramov:1989} to decide whether \eqref{wrat-check} admits a solution in $\bar{\mathbb{Q}}(x)$.

If \eqref{wrat-check} does not admit a solution in $\bar{\mathbb{Q}}(x)$, then $R_u(G^0)=R_u(H^0)$ by Proposition~\ref{bigut2}. Since $R_u(G)=R_u(G^0)$ and $R_u(H^0)=R_u(H)=\mathbb{G}_a(C)$, it follows that $G$ coincides with the subgroup of \begin{equation}\label{ut-last}\mathbb{G}_m(C)^2\ltimes\mathbb{G}_a(C)=\left\{\begin{pmatrix} \alpha & \beta \\ 0 & \lambda \end{pmatrix} \ \middle| \ \alpha,\beta,\lambda\in C, \ \alpha\lambda\neq 0\right\}\end{equation} defined by the conditions on $\alpha$ and $\gamma$ that we compute via Proposition~\ref{galdiag} and Remark~\ref{diagonalizable-rem}.

If \eqref{wrat-check} does admit a solution $w_t\in\bar{\mathbb{Q}}(x)$, then we have that $H^0$ is a connected subgroup of $\mathbb{G}_m(C)\ltimes\mathbb{G}_a(C)$ as in \eqref{bigut}, and we then verify which of the conditions in Proposition~\ref{wrat} holds. Conditions (i) and (ii) are verified just as in Remark~\ref{diagonalizable-rem}, bearing in mind that the $u$ and $v$ referred to there are both replaced with $u_t$ here. If any of these conditions are satisfied, so that $R_u(G)=R_u(G^0)=\mathbb{G}_a(C)$, then we again have that $G$ coincides with the subgroup of $\mathbb{G}_m(C)^2\ltimes\mathbb{G}_a(C)$ defined by the conditions on $\alpha$ and $\lambda$ that we compute via Proposition~\ref{galdiag} and Remark~\ref{diagonalizable-rem}.

In order to verify the conditions in Proposition~\ref{wrat}(iii), we proceed as follows: for each $\mathbb{Z}$-orbit $[d]$ for $d\in \bar{\mathbb{Q}}$, and each $n\in\mathbb{N}$, we compute the discrete residues $\mathrm{dres}_x(\frac{\delta(u_t)}{u_t},[d],1)$ and $\mathrm{dres}_x(w_t,[d],n)$. Let $r\in\mathbb{N}$ denote the largest degree such that $\mathrm{dres}_x(w_t,[d],r)\neq 0$ for some $[d]$. Then we may write the linear $\delta$-polynomial $\mathcal{L}\in\bar{\mathbb{Q}}\{Y\}_\delta$ of Propostition~\ref{wrat}(iii) with undetermined coefficients as $\mathcal{L}=\sum_{i=0}^{r-1}c_i\delta^i(Y)$, and compute the coefficients $c_i$ from the relations \begin{equation}\label{galdiag-syst}(-1)^ii!\mathrm{dres}_x\left(\frac{\delta(u_t)}{u_t},[d],i+1\right)c_i=\mathrm{dres}_x(w_t,[d],i+1),\end{equation} which must be satisfied for every $0\leq i \leq r-1$ and every $[d]$ simultaneously.

If the system \eqref{galdiag-syst} does not admit a solution $(c_0,\dots,c_{r-1})\in\bar{\mathbb{Q}}^r$, then $G^0$ and $H^0$ both coincide with $\mathbb{G}_m(C)\ltimes\mathbb{G}_a(C)$ as in \eqref{bigut}, and in this case $G$ is the subgroup of $\mathbb{G}_m(C)^2\ltimes\mathbb{G}_a(C)$ defined by the conditions on $\alpha$ and $\lambda$ that we compute via Proposition~\ref{galdiag} and Remark~\ref{diagonalizable-rem}. If the system \eqref{galdiag-syst} does admit a solution $(c_0,\dots,c_{r-1})\in\bar{\mathbb{Q}}^r$, then $G^0$ is as described in Proposition~\ref{wrat}(iii), with $\mathcal{L}=\sum_{i=0}^{r-1}c_i\delta^i(Y)$. \end{rem}

\section{Imprimitive groups}\label{dihedral-sec}

We recall the notation introduced in the previous sections: $k=C(x)$, where $C$ is a $\delta$-closure of $\bar{\mathbb{Q}}$, $\sigma$ denotes the $C$-linear automorphism of $k$ defined by $\sigma(x)=x+1$, and $\delta(x)=1$. Let us now suppose that there are no solutions in $\bar{\mathbb{Q}}(x)$ to the first Riccati equation \eqref{ric1}, and that either $a=0$ or else there exists a solution $e\in\bar{\mathbb{Q}}(x)$ to the second Riccati equation \eqref{ric2}, so that the matrix $\sigma$-equation \eqref{mateq} associated to \eqref{difeq} is equivalent (in the sense of Definition~\ref{equivalent-def}) to \begin{equation} \label{dih-mateq} \sigma(Y)=\begin{pmatrix} 0 & 1 \\ -r & 0 \end{pmatrix}Y\end{equation} for some $r\in\bar{\mathbb{Q}}(x)$, as in \eqref{impequiv}. Then the $\sigma\delta$-Galois group $G$ for \eqref{difeq} is conjugate to a non-diagonal subgroup of \eqref{dihedral}, by Proposition~\ref{dense} and the results of \cite{hendriks:1998} summarized in \S\ref{hendriks-sec}. In particular, $G$ has at least two connected components. More precisely, if we let $$\mathrm{Diag}=\left\{\begin{pmatrix} \alpha & 0 \\ 0 & \lambda\end{pmatrix} \ \middle| \ \alpha,\lambda\in C, \ \alpha\lambda\neq 0\right\},$$ then the connected component of the identity $G^0\subseteq\mathrm{Diag}\cap G$, and we have an exact sequence $$\{1\}\longrightarrow \mathrm{Diag}\cap G \longrightarrow G \longrightarrow \{\pm 1\}\longrightarrow \{1\}.$$ Thus, the computation of $G$ is (almost) reduced to the computation of the diagonal group $\mathrm{Diag}\cap G$.

By Proposition~\ref{idemp}, the $\sigma\delta$-PV ring $R$ for \eqref{dih-mateq} decomposes into a direct product $R=R_0\oplus R_1$, where each $R_i$ is a $\sigma^2\delta$-PV ring over $k$ for \begin{equation} \label{dih-mateq2} \sigma^2(Y) =\begin{pmatrix} 0 & 1 \\ \sigma(-r) & 0 \end{pmatrix}\begin{pmatrix} 0 & 1 \\ -r & 0 \end{pmatrix}Y=\begin{pmatrix} -r & 0 \\ 0 & -\sigma(r)\end{pmatrix}Y.\end{equation} Moreover, $\mathrm{Diag}\cap G$ is identified with $\mathrm{Gal}_{\sigma^2\delta}(R_i/k)$ for $i=0,1$.

If we write $\tilde{\sigma}=\sigma^2$, $\tilde{x}=\frac{x}{2}$, and $\tilde{\delta}=2\delta$, then $\bar{\mathbb{Q}}(x)=\bar{\mathbb{Q}}(\tilde{x})$, $k=C(\tilde{x})$, $\tilde{\sigma}(\tilde{x})=\tilde{x}+1$, and $\tilde{\delta}(\tilde{x})=1$ (cf.~the last paragraph of Remark~\ref{connected-rem}). If we consider $r$ and $\sigma(r)$ as rational functions in $\tilde{x}$ with coefficients in $\bar{\mathbb{Q}}$, we may compute the $\tilde{\sigma}\tilde{\delta}$-Galois group $\mathrm{Gal}_{\tilde{\sigma}\tilde{\delta}}(R_0/k)=G^0$ for $$\tilde{\sigma}(Y)=\begin{pmatrix} -r & 0 \\ 0 & -\sigma(r)\end{pmatrix}Y$$ as in Proposition~\ref{galdiag}, with $u=-r$ and $v=-\sigma(r)$. However, not every case of Proposition~\ref{galdiag} can occur.

\begin{prop} \label{dihedral-complete} Suppose that $0\neq r\in\bar{\mathbb{Q}}(x)$, and suppose that the $\sigma$-Galois group $H$ for \begin{equation}\label{dih-system}\sigma(Y)=\begin{pmatrix} 0 & 1 \\ -r & 0\end{pmatrix}Y,\end{equation} is irreducible and imprimitive. Then $G$ is the subgroup of \begin{equation}\label{dih-gp-prop}\{\pm 1\}\ltimes\mathbb{G}_m(C)^2=\left\{\begin{pmatrix} \alpha & 0 \\ 0 & \lambda\end{pmatrix} \ \middle| \ \alpha,\lambda\in C, \ \alpha\lambda\neq 0\right\} \cup\left\{\begin{pmatrix} 0 & \beta \\ \epsilon & 0\end{pmatrix} \ \middle| \ \beta,\epsilon\in C, \ \beta\epsilon\neq 0\right\}\end{equation} defined by the following conditions on $\alpha$, $\lambda$, $\beta$, and $\epsilon$.
\begin{enumerate}
\item[(i)] There exists a nonzero element $f\in\bar{\mathbb{Q}}(x)$ and a primitive $m^\text{th}$ root of unity $\zeta_m$ such that $r=\zeta_m\frac{\sigma(f)}{f}$ if and only if $\mathrm{det}(G)=\mu_m$, the group of $m^\text{th}$ roots of unity, or equivalently:
\begin{enumerate}
\item[(a)] when $m$ is odd, $(\alpha\lambda)^m=1$ and $(\beta\epsilon)^m=-1$.
\item[(b)] when $m$ is divisible by $4$, $(\alpha\lambda)^{\frac{m}{2}}=1$ and $(\beta\epsilon)^{\frac{m}{2}}=-1$.
\item[(c)] when $m$ is even but not divisible by $4$, $(\alpha\lambda)^{\frac{m}{2}}=1$ and $(\beta\epsilon)^{\frac{m}{2}}=1$.
\end{enumerate}
\item[(ii)] Case (i) does not hold and there exists an element $f\in\bar{\mathbb{Q}}(x)$ such that $\frac{\delta(r)}{r}=\sigma(f)-f$ if and only if $\mathrm{det}(G)=\mathbb{G}_m(C^\delta)$, or equivalently $\delta(\alpha\lambda)=\delta(\beta\epsilon)=0$.
\item[(iii)] If neither (i) nor (ii) holds, then $G=\{\pm 1\}\ltimes\mathbb{G}_m(C)^2$.
\end{enumerate}
\end{prop}

\begin{proof} Let us write $T_\gamma\in\mathrm{GL}_2(C)$ for the matrix corresponding to $\gamma\in G$, so that $$T_\gamma=\begin{pmatrix} \alpha_\gamma & 0 \\ 0 & \lambda_\gamma\end{pmatrix} \qquad \text{or}\qquad T_\gamma=\begin{pmatrix} 0 & \beta_\gamma \\ \epsilon_\gamma & 0\end{pmatrix},$$ depending on whether $\gamma\in \mathrm{Diag}\cap G$, the connected component of the identity, or $T_\gamma\in G-\mathrm{Diag}\cap G$, the complement of $\mathrm{Diag}\cap G$ in $G$. Since $\mathrm{det}(G)$ and $\mathrm{det}(H)$ are the $\sigma\delta$-Galois group and $\sigma$-Galois group, respectively, for $\sigma(y)=ry$, it follows from Proposition~\ref{dense} that $\mathrm{det}(H)$ is finite if and only if $\mathrm{det}(G)$ is finite, and in this case $\mathrm{det}(G)=\mathrm{det}(H)$. It is proved in \cite[Lem.~4.8]{hendriks:1998} that either $\mathrm{det}(H)$ is infinite, or else $\mathrm{Diag}\cap H=\left\{\left(\begin{smallmatrix} \alpha & 0 \\ 0 & \lambda\end{smallmatrix}\right) \ \middle| \ (\alpha\lambda)^n=1\right\}$ for some positive integer $n$. In the latter case, if $n$ is even then $(\beta_\gamma\epsilon_\gamma)^n=-1$ for every $\gamma\in H$, and if $n$ is odd either  $(\beta_\gamma\epsilon_\gamma)^n=-1$ for every $\gamma\in H$ or $(\beta_\gamma\epsilon_\gamma)^n=1$ for every $\gamma\in H$.

So let us first assume that $D=\mathrm{det}(G)=\mathrm{det}(H)$ is finite and decide which of these possibilities occurs. It follows from \cite[Cor.~2.3]{vanderput-singer:1997} that $D$ is finite of order $m$ if and only if there exists $f\in\bar{\mathbb{Q}}(x)$ such that $r=\zeta_m\frac{\sigma(f)}{f}$ for some primitive $m^\text{th}$ root of unity $\zeta_m$ as in part (i). In any case, we see that $D=\{\alpha_\gamma\lambda_\gamma, -\beta_\gamma\epsilon_\gamma \ | \ \gamma\in G\}$, and therefore $(\alpha_\gamma\lambda_\gamma)^m=1$ for each $\gamma\in \mathrm{Diag}\cap G$ and $(-\beta_\gamma\epsilon_\gamma)^m=1$ for every $\gamma\in G-\mathrm{Diag}\cap G$. If $m$ is odd, this implies that $(\beta_\gamma\epsilon_\gamma)^m=-1$ for every $\gamma\in G-\mathrm{Diag}\cap G$, so $m=n$ and we are in case (a) of part (i). If $m=2n$ is even, then we see that $(\beta_\gamma\epsilon_\gamma)^n=\pm 1$ for every $\gamma\in G-\mathrm{Diag}\cap G$, which implies that $(\alpha_\gamma\lambda_\gamma)^n=1$ for every $\gamma\in \mathrm{Diag}\cap G$. In this case, if $(\beta_\gamma\epsilon_\gamma)^n=-1$ for every $\gamma\in G-\mathrm{Diag}\cap G$, then $n$ must be even, for otherwise $D$ would have order $n$, not $m=2n$, and we see that we are in case (b) of part (i). Finally, if $(\beta_\gamma\epsilon_\gamma)^n=1$ for every $\gamma\in G-\mathrm{Diag}\cap G$, then $n$ must be odd by \cite[Lem.~4.8]{hendriks:1998}, and we are in case (c) of part (i). This concludes the proof of part (i) and the computation of $G$ when $\mathrm{det}(G)$ is assumed to be finite.

Now suppose that $D=\mathrm{det}(G)$ is infinite, so that the $\sigma$-Galois group $H$ for \eqref{dih-system} coincides with \eqref{dih-gp-prop} by \cite[Lem.~4.8]{hendriks:1998}. Note that in this case $\mathrm{Diag}\cap G=G^0$, the connected component of the identity. By Lemma~\ref{idemp}, $G^0$ is identified with the $\sigma^2\delta$-Galois group for the system \begin{equation} \label{dih-conn} \sigma^2(Y)=\begin{pmatrix} -r & 0 \\ 0 & -\sigma(r)\end{pmatrix}Y.\end{equation} We will compute $G^0$ with Proposition~\ref{galdiag}, after replacing $\sigma$ with $\sigma^2$ as we explained above, with $u=-r$ and $v=-\sigma(r)$. We have already dealt with cases (i) and (ii) of Proposition~\ref{galdiag} (to wit: case (i) does not occur, and case (ii) only occurs with $m=n$; cf.~\cite[Lem.~4.8]{hendriks:1998}).

Suppose that $G^0$ is as in Proposition~\ref{galdiag}(iii), so that either there exists $f\in\bar{\mathbb{Q}}(x)$ with $\frac{\delta(r)}{r}=\sigma^2(f)-f$, or there exists $f'\in\bar{\mathbb{Q}}(x)$ such that $\frac{\delta(\sigma(r))}{\sigma(r)}=\sigma^2(f')-f'$. Then we see that both conditions are satisfied simultaneously, by setting $f'=\sigma(f)$ and $f=\sigma^{-1}(f')$. It follows from \cite[Cor.~3.4(1)]{hardouin-singer:2008} that $r=c\frac{\sigma^2(g)}{g}$ for some $c\in\bar{\mathbb{Q}}$ and $g\in\bar{\mathbb{Q}}(x)$. If we let $u=\sqrt{-c}\frac{\sigma(g)}{g}$, we see that both $u$ and $-u$ satisfy the Riccati equation \eqref{ric1}, contradicting the hypothesis that $H$ is irreducible. So $G^0$ cannot be as described in Proposition~\ref{galdiag}(iii).

Thus far, we have considered cases (i)--(iii) of Proposition~\ref{galdiag} (applied to \eqref{dih-conn}, where we consider $k$ as a $\sigma^2\delta$-field, with $u=-r$ and $v=-\sigma(r)$). Let us now suppose that $G^0$ is as in Proposition~\ref{galdiag}(iv), so that there exist nonzero integers $m$ and $n$ such that $m\frac{\delta(r)}{r}+n\frac{\delta(\sigma(r))}{\sigma(r)}=\sigma^2(g)-g$ for some $g\in\bar{\mathbb{Q}}(x)$, which occurs if and only if $\delta(\alpha_\gamma^m\lambda_\gamma^n)=0$ for every $\gamma\in G^0$. But then for every $\gamma\in G^0$ and any $\tau\in G-G^0$, we have $\tau\gamma\tau^{-1}\in G^0$, and $\alpha_{\tau\gamma\tau^{-1}}=\lambda_\gamma$ and $\lambda_{\tau\gamma\tau^{-1}}=\alpha_\gamma$, which implies that $\delta(\alpha_\gamma^n\lambda_\gamma^m)=0$ for every $\gamma\in G^0$, which again by Proposition~\ref{galdiag}(iv) implies that $n\frac{\delta(r)}{r}+m\frac{\delta(\sigma(r))}{\sigma(r)}=\sigma^2(\tilde{g})-\tilde{g}$ for some $\tilde{g}\in\bar{\mathbb{Q}}(x)$. If we let $f=(g+\tilde{g})/(m+n)$, we see that $\frac{\delta(r)}{r}+\frac{\delta(\sigma(r))}{\sigma(r)}=\sigma^2(f)-f$. We claim that $\frac{\delta(r)}{r}=\sigma(f)-f$. To see this, observe that $$\sigma\left(\sigma(f)-f-\frac{\delta(r)}{r}\right)=\sigma^2(f)-\sigma(f)-\frac{\delta(\sigma(r))}{\sigma(r)}=\frac{\delta(r)}{r}+f-\sigma(f)=-\left(\sigma(f)-f-\frac{\delta(r)}{r}\right).$$ By \cite[Cor.~2.3(2)]{vanderput-singer:1997}, the $\sigma$-Galois group over $\bar{\mathbb{Q}}(x)$ for $\sigma(y)=-y$ is $\{\pm 1\}$, whence the only element of $\bar{\mathbb{Q}}(x)$ that satisfies $\sigma(y)=-y$ is $y=0$, and therefore we are in case (iii) of Proposition~\ref{dihedral-complete}. On the other hand, if we assume that $\frac{\delta(r)}{r}=\sigma(f)-f$ for some $f\in\bar{\mathbb{Q}}(x)$, then the equation $\sigma(y)=ry$ has $\sigma\delta$-Galois group $D\subseteq\mathbb{G}_m(C^\delta)$ by Proposition~\ref{dense}, which implies that $\delta(\mathrm{det}(T_\gamma))=0$ for every $\gamma\in G$.

Finally, suppose that $G^0$ is as in Proposition~\ref{galdiag}(v), so that $G^0$ is the full diagonal group $\left\{\left(\begin{smallmatrix} \alpha & 0 \\ 0 & \lambda\end{smallmatrix}\right) \ \middle| \ \alpha\lambda\neq 0 \right\}$, which implies that $H^0=G^0$. Since $G$ is irreducible, there exists some $\left(\begin{smallmatrix} 0 & \beta \\ \epsilon & 0\end{smallmatrix}\right)\in G$. Since $G\subseteq H$ and $G/G^0$ coincides with $H/H^0$, it follows that $G=H$. This finishes the proof. \end{proof}

\begin{rem} To compute the difference-differential Galois group $G$ for \eqref{difeq} when there is no solution $u\in\bar{\mathbb{Q}}(x)$ to the Riccati equation \eqref{ric1}, we proceed as follows. If $a=0$, we let $r=b$. If $a\neq 0$, we apply the results of \cite{hendriks:1998} to compute a solution $e\in\bar{\mathbb{Q}}(x)$ to the Riccati equation \eqref{ric2}, assuming it exists (otherwise, $\mathrm{SL}_2(C)\subseteq H$ and the computation of $G$ is carried out in the next section), and set $r=-a\sigma(a)+\sigma(b)+a\sigma^2(\frac{b}{a})+a\sigma^2(e)$.

In order to verify the conditions of Proposition~\ref{dihedral-complete}, we proceed as follows. We first compute the discrete residues $q_{[d]}=\mathrm{dres}_x(\frac{\delta(r)}{r},[d],1)$ at each $\mathbb{Z}$-orbit $[d]$ for $d\in \bar{\mathbb{Q}}$. By \cite[Lem.~2.1]{vanderput-singer:1997} condition (i) occurs when $q_{[d]}=0$ for every $[d]$ and $r(\infty)=\zeta_m$. By Proposition~\ref{indsum}, condition (ii) occurs when $q_{[d]}=0$ for every $[d]$. 
\end{rem}

\begin{rem} We refer to \cite[\S4]{hendriks-singer:1999} for a general discussion, in the context of the classical $\sigma$-Picard-Vessiot theory of \cite{vanderput-singer:1997}, of a family of difference equations that includes \eqref{dih-mateq} as a special case, and of how this family is related to the phenomenon of \emph{interlacing}, which is one of the most striking differences between the Picard-Vessiot theories for differential equations \cite{vanderput-singer:2003} and for difference equations \cite{vanderput-singer:1997}.
\end{rem}

\section{Groups containing $\mathrm{SL}_2$} \label{large-sec}

We recall the notation introduced in the previous sections: $k=C(x)$, where $C$ is a $\delta$-closure of $\bar{\mathbb{Q}}$, $\sigma$ denotes the $C$-linear automorphism of $k$ defined by $\sigma(x)=x+1$, and $\delta(x)=1$. We write $H$ for the $\sigma$-Galois group and $G$ for the $\sigma\delta$-Galois group for \begin{equation}\label{difeqsl2}\sigma(Y)=\begin{pmatrix} 0 & 1 \\ -b & -a\end{pmatrix}Y\end{equation} over $k$, where $a,b\in\bar{\mathbb{Q}}(x)$ and $b\neq 0$. In this section we consider the case where $a\neq 0$ and there are no solutions in $\bar{\mathbb{Q}}(x)$ to \eqref{ric1} nor to \eqref{ric2}, which is equivalent to the condition that $\mathrm{SL}_2(C)\subseteq H$ by the results of \cite{hendriks:1998} summarized in \S\ref{hendriks-sec}.

The following result will allow us to reduce the computation of $G$ in Theorem~\ref{g-large} to the computation of the $\sigma\delta$-Galois group $\mathrm{det}(G)$ for the first-order equation $\sigma(y)=by$.

\begin{prop}\label{unimod-dense} If $\mathrm{SL}_2(C)\subseteq H$, then $\mathrm{SL}_2(C)\subseteq G$.\end{prop}

\begin{proof} We begin by showing that $G\cap\mathrm{SL}_2(C)$ is Zariski-dense in $\mathrm{SL}_2(C)$. We denote the Zariski closure of a subset $V\subseteq\mathrm{GL}_2(C)$ by $\overline{V}$. By Proposition~\ref{dense}, $\overline{G}=H$. Let us first assume that $G$ is connected, so $H$ is also connected by \cite[Cor.~3.7]{minchenko-ovchinnikov-singer:2013b}. Since $G\cap\mathrm{SL}_2(C)$ is normal in $G$, by \cite[Lem.~3.8]{minchenko-ovchinnikov:2011} $\overline{G\cap\mathrm{SL}_2(C)}$ is a normal algebraic subgroup of $\overline{G}$. Hence, $\overline{G\cap\mathrm{SL}_2(C)}$ is also a normal algebraic subgroup of $\mathrm{SL}_2(C)\subseteq H= \overline{G}$. Hence, $\overline{G\cap\mathrm{SL}_2(C)}$ is either $\mathrm{SL}_2(C)$ or a subgroup of $\{\pm 1\}$.

We proceed by contradiction, and assume that $\overline{G\cap\mathrm{SL}_2(C)}\subseteq\{\pm 1\}$. Then $G\cap\mathrm{SL}_2(C)\subseteq\{\pm 1\}$. Since the quotient $G/(G\cap\mathrm{SL}_2(C))\simeq\mathrm{det}(G)$ is commutative, the image $[G,G]$ of the commutator map $G\times G\rightarrow G:(g,h)\mapsto ghg^{-1}h^{-1}$ is contained in $G\cap\mathrm{SL}_2(C)\subseteq\{\pm 1\}$. But since $G$ is connected and the commutator map is continuous, this image is also connected, so $[G,G]=\{1\}$ and therefore $G$ is commutative. But this would imply that $\overline{G}=H$ is also commutative (since commutativity is a closed condition), contradicting the assumption that $\mathrm{SL}_2(C)\subseteq H$. We have just shown that if $G$ is connected, then $G\cap\mathrm{SL}_2(C)$ is Zariski-dense in $\mathrm{SL}_2(C)$.

If $G$ is not connected, let $G^0$ denote the connected component of the identity in $G$. It follows from \cite[proof of Cor.~3.7]{minchenko-ovchinnikov-singer:2013b} that $\overline{G^0}=H^0$. Now $\mathrm{SL}_2(C)\subseteq\overline{G^0}$, since $\mathrm{SL}_2(C)$ is connected, and the argument above shows that $G^0\cap\mathrm{SL}_2(C)$ is Zariski-dense in $\mathrm{SL}_2(C)$, concluding the proof $G\cap\mathrm{SL}_2(C)$ is Zariski-dense in $\mathrm{SL}_2(C)$.

By \cite[Prop.~42]{cassidy:1972}, a Zariski-dense differential-algebraic subgroup of $\mathrm{SL}_2(C)$ is either $\mathrm{SL}_2(C)$ or conjugate to $\mathrm{SL}_2(C^\delta)$. To see that $G\cap\mathrm{SL}_2(C)$ cannot be conjugate to $\mathrm{SL}_2(C^\delta)$, let us assume without loss of generality that $G\cap\mathrm{SL}_2(C)=\mathrm{SL}_2(C^\delta)$ and obtain a contradiction. Since $G$ normalizes $\mathrm{SL}_2(C^\delta)$, \cite[Lem.~11]{mitschi-singer:2012a} implies that \begin{equation}\label{projconst}G\subset \mathrm{GL}_2(C^\delta)\cdot\mathrm{Scal}(C)=\mathrm{SL}_2(C^\delta)\cdot\mathrm{Scal}(C),\end{equation} where $\mathrm{Scal}(C)$ denotes the group of scalar matrices in $\mathrm{GL}_2(C)$. This means that $G$ is projectively $\delta$-constant in the sense of \cite{arreche-singer:2016}, where it is shown that \eqref{projconst} implies that $H$ and $G$ are solvable, contradicting our assumption that $\mathrm{SL}_2(C)\subseteq H$. This concludes the proof that $\mathrm{SL}_2(C)\subseteq G$. \end{proof}

\begin{thm} \label{g-large} If $\mathrm{SL}_2(C)\subseteq H$, then $G$ is the subgroup of $\mathrm{GL}_2(C)$ defined by one of the following conditions. \begin{enumerate}
\item[(i)] There exists a nonzero element $f\in\bar{\mathbb{Q}}(x)$ and a primitive $m^\text{th}$ root of unity $\zeta_m$ such that $b=\zeta_m\frac{\sigma(f)}{f}$ if and only if $G=H=\{T\in\mathrm{GL}_2(C) \ | \ \mathrm{det}(T)^m=1\}$.
\item[(ii)] Case (i) does not hold and there exists an element $f\in\bar{\mathbb{Q}}(x)$ such that $\sigma(f)-f=\frac{\delta(b)}{b}$ if and only if $G=\{T\in\mathrm{GL}_2(C) \ | \ \delta(\mathrm{det}(T))=0\}$.
\item[(iii)] If neither (i) nor (ii) holds, then $G=H=\mathrm{GL}_2(C)$.
\end{enumerate}
\end{thm}

\begin{proof} Let $D\subseteq\mathrm{Scal}(C)$ denote the subgroup of scalar matrices in $\mathrm{GL}_2(C)$ whose determinant lies in the subgroup $\mathrm{det}(G)\subseteq\mathbb{G}_m(C)$. By Proposition~\ref{unimod-dense}, we know that $\mathrm{SL}_2(C)\subseteq G$, which implies that $G\subseteq \mathrm{SL}_2(C)\cdot D$. Since the images of $G$ and $\mathrm{SL}_2(C)\cdot D$ under the determinant map are equal, we have that $G=\mathrm{SL}_2(C)\cdot D$. Therefore, the computation of $G$ is reduced to that of $\mathrm{det}(G)$.

Since $\mathrm{det}(H)$ is the $\sigma$-Galois group and $\mathrm{det}(G)$ is the $\sigma\delta$-Galois group for $\sigma(y)=by$ over $k$, it follows from \cite[\S3]{hendriks:1998} that $\mathrm{det}(H)$ is finite of order $m$ if and only if $b=\zeta_m\frac{\sigma(f)}{f}$ for some $f\in\bar{\mathbb{Q}}(x)$ and some primitive root of unity $\zeta_m$, in which case $\mathrm{det}(G)=\mathrm{det}(H)$ by Proposition~\ref{dense}, which implies (i).

To show (ii), we observe that $\sigma(f)=f+\frac{\delta(b)}{b}$ is precisely the integrability condition for the first-order system $\sigma(y)=by$, and we conclude by Proposition~\ref{dconst}.

To show (iii), let $Z$ be a fundamental solution matrix for \eqref{difeqsl2} and $z=\mathrm{det}(Z)$ the Casoratian determinant, so that $\sigma(z)=bz$. If condition (ii) fails, it follows from \cite[Lem.~2.4]{hardouin:2008} (cf.~Remark~\ref{c-ok}) that there is no $f\in k$ such that $\frac{\delta(b)}{b}=\sigma(f)-f$, whence $z$ is $\delta$-transcendental over $k$ by Corollary~\ref{diffprod}. By \cite[Prop.~6.26]{hardouin-singer:2008}, this implies that the $\delta$-dimension of $\mathrm{det}(G)$ is $1$, and therefore $\mathrm{det}(G)=\mathbb{G}_m(C)$. \end{proof}

\begin{rem} In order to compute $G$ when $\mathrm{SL}_2(C)\subseteq H$, we proceed as follows. We begin by computing the discrete residues $r_{[d]}=\mathrm{dres}_x(\frac{\delta(b)}{b},[d],1)$ at each $\mathbb{Z}$-orbit $[d]$ for $d\in \bar{\mathbb{Q}}$. By \cite[Lem.~2.1]{vanderput-singer:1997}, condition (i) occurs when $r_{[d]}=0$ for every $[d]$ and $r(\infty)=\zeta_m$. By Proposition~\ref{indsum}, condition (ii) occurs when $r_{[d]}=0$ for every $[d]$.
\end{rem}

\section{From $\sigma\delta$-Galois groups to $\sigma\delta$-algebraic relations} \label{relations-sec}

In this section we assume that the $\sigma\delta$-Galois group $G\subseteq\mathrm{GL}_2(C)$ for \begin{equation} \label{relations-eq} \sigma^2(y)+a\sigma(y)+by=0\end{equation} has already been computed, and show how to obtain explicitly all the $\sigma\delta$-relations among the solutions to \eqref{relations-eq} from the knowledge of $G$. Here we aim to convince the nonexpert that, once $G$ has been computed, it is straightforward to compute the $\sigma\delta$-algebraic relations among the solutions without any knowledge of $\sigma\delta$-Galois theory. More specifically, for each possible defining equation for $G$ we will deduce a relation of a specific form that is satisfied by the solutions to \eqref{relations-eq}, but with an undetermined coefficient. Each undetermined coefficient is obtained by finding all solutions in $\bar{\mathbb{Q}}(x)$ to a concrete $\sigma$-equation defined over $\bar{\mathbb{Q}}(x)$.

Recall that the coefficients $a,b\in\bar{\mathbb{Q}}(x)$, the $\bar{\mathbb{Q}}$-linear automorphism $\sigma:x\mapsto x+1$, $\delta(x)=1$, and $C$ denotes a $\delta$-closure of $\bar{\mathbb{Q}}$. We denote by $k=C(x)$ the $\sigma\delta$-field ring over $\bar{\mathbb{Q}}(x)$ obtained by setting $\sigma|_C=\mathrm{id}_C$. Given a basis of solutions $\{y_1,y_2\}$ for \eqref{relations-eq}, it will be convenient to consider the fundamental solution matrix \begin{equation}\label{rel-fund}Z=\begin{pmatrix} y_1 & y_2 \\ \sigma(y_1) & \sigma(y_2)\end{pmatrix}, \quad \text{which satisfies} \quad \sigma(Z)=\begin{pmatrix} 0 & 1 \\ -b & -a\end{pmatrix}Z.\end{equation}

\subsection{$G$ is diagonalizable} \label{diag-subsec}

Let us first assume that $G$ is diagonalizable as in \S\ref{diagonalizable-sec}, and the embedding $G\hookrightarrow\mathrm{GL}_2(C)$ corresponding to the fundamental solution matrix \eqref{rel-fund} is given by \begin{equation}\label{diag-emb} \gamma\left(\begin{pmatrix} y_1& y_2 \\ \sigma(y_1) & \sigma(y_2)\end{pmatrix}\right)=\begin{pmatrix} y_1& y_2 \\ \sigma(y_1) & \sigma(y_2)\end{pmatrix}\begin{pmatrix} \alpha_\gamma & 0 \\ 0 & \lambda_\gamma\end{pmatrix}=\begin{pmatrix} \alpha_\gamma y_1 & \lambda_\gamma y_2\\ \alpha_\gamma\sigma(y_1) & \lambda_\gamma\sigma(y_2)\end{pmatrix} \end{equation} for some $\alpha_\gamma,\lambda_\gamma\in C^\times$. Since $$\gamma\left(\frac{\sigma(y_1)}{y_1}\right)=\frac{\alpha_\gamma\sigma(y_1)}{\alpha_\gamma y_1}=\frac{\sigma(y_1)}{y_1} \qquad \text{and} \qquad \gamma\left(\frac{\sigma(y_2)}{y_2}\right)=\frac{\lambda_\gamma\sigma(y_2)}{\lambda_\gamma y_2}=\frac{\sigma(y_2)}{y_2}$$ for every $\gamma\in G$, it follows from Theorem~\ref{correspondence} that $\frac{\sigma(y_1)}{y_1}=u\in k$ and $\frac{\sigma(y_2)}{y_2}=v\in k$.

Since $G$ is Zariski dense in the $\sigma$-Galois group $H$ associated to \eqref{relations-eq}, if $G$ is diagonalizable then so is $H$. As we discussed in \S\ref{hendriks-sec}, it was already proved in \cite{hendriks:1998} that $H$ is diagonalizable if and only if there exists a basis of solutions $\{y_1,y_2\} $ for \eqref{relations-eq} that satisfy $\sigma(y_1)=uy_1$ and $\sigma(y_2)=vy_2$ for some $u,v\in\bar{\mathbb{Q}}(x)$, where the coefficients $u$ and $v$ are distinct solutions to the Riccati equation \eqref{ric1}: $$u\sigma(u)+au+b=0=v\sigma(v)+av+b,$$ and the computation of explicit coefficients $u,v\in\bar{\mathbb{Q}}(x)$ as above is carried out in \cite[pp.~450--451]{hendriks:1998}. 

Let us now assume that there exist integers $m$ and $n$ such that $\alpha_\gamma^m\lambda_\gamma^n=1$ as in Proposition~\ref{galdiag}(ii). Then $\gamma(y_1^my_2^n)=\alpha_\gamma^m\lambda_\gamma^n(y_1^my_2^n)=y_1^my_2^n$, and therefore $y_1^my_2^n=f\in k$ by Theorem~\ref{correspondence}. The coefficient $f$ is obtained by solving $\sigma(f)=u^mv^nf$ in $\bar{\mathbb{Q}}(x)$ \cite[Lem.~2.5]{hardouin:2008} (cf.~Remark~\ref{c-ok}). 

Let us now assume that $\delta(\alpha_\gamma)=0$ for every $\gamma\in G$ as in Proposition~\ref{galdiag}(iii). Then, since $\gamma(\frac{\delta(y_1)}{y_1})=\frac{\delta(y_1)}{y_1}+\frac{\delta(\alpha_\gamma)}{\alpha_\gamma}=\frac{\delta(y_1)}{y_1}$, and therefore $\frac{\delta(y_1)}{y_1}=f\in k$ by Theorem~\ref{correspondence}. The coefficient $f$ is obtained by solving $\sigma(f)-f=\frac{\delta(u)}{u}$ in $\bar{\mathbb{Q}}(x)$ \cite[Lem.~2.4]{hardouin:2008} (cf.~Remark~\ref{c-ok}). A similar argument shows that if $\delta(\lambda_\gamma)=0$ for every $\gamma\in G$ as in Proposition~\ref{galdiag}(iii), then $\delta(y_2)=fy_2$, where $f\in\bar{\mathbb{Q}}(x)$ satisfies $\sigma(f)-f=\frac{\delta(v)}{v}$.

Let us now assume that Proposition~\ref{galdiag}(ii) does not hold, but there do exist integers $m$ and $n$ as in Proposition~\ref{galdiag}(iv), such that $\delta(\alpha_\gamma^m\lambda_\gamma^n)=0$ for every $\gamma\in G$. Then, since $$\gamma\left(\frac{\delta(y_1^my_2^n)}{y_1^my_2^n}\right)=\frac{\delta(\alpha_\gamma^m\lambda_\gamma^ny_1^my_2^n)}{\alpha_\gamma^m\lambda_\gamma^ny_1^my_2^n}=\frac{\delta(y_1^my_2^n)}{y_1^my_2^n}+\frac{\delta(\alpha_\gamma^m\lambda_\gamma^n)}{\alpha_\gamma^m\lambda_\gamma^n}=\frac{\delta(y_1^my_2^n)}{y_1^my_2^n},$$ it follows from Theorem~\ref{correspondence} that $\delta(y_1^my_2^n)/y_1^my_2^n=f\in k$. The coefficient $f$ in the relation $\delta(y_1^my_2^n)=fy_1^my_2^n$ is obtained by solving $\sigma(f)-f=\delta(u^mv^n)/u^mv^n$ in $\bar{\mathbb{Q}}(x)$ \cite[Lem.~2.4]{hardouin:2008} (cf.~Remark~\ref{c-ok}).

Let us now assume that $G=\mathbb{G}_m(C)$ as in Proposition~\ref{galdiag}(v). Then by \cite[Lem.~6.26]{hardouin-singer:2008}, the $\delta$-transcendence degree of the $\sigma\delta$-PV ring $k\{y_1,y_2, (y_1y_2)^{-1}\}_\delta$ for \eqref{relations-eq} over $k$ is $2$, and therefore $y_1$ and $y_2$ are $\delta$-independent over $k$.

\subsection{$G$ is reducible} Let us now assume that $G$ is reducible but not diagonalizable as in \S\ref{reducible-sec}, and that the embedding $G\hookrightarrow \mathrm{GL}_2(C)$ corresponding to the fundamental solution matrix \eqref{rel-fund} is given by \begin{equation}\label{red-emb} \gamma\left(  \begin{pmatrix} y_1 & y_2 \\ \sigma(y_1) & \sigma(y_2)\end{pmatrix}  \right)= \begin{pmatrix} y_1 & y_2 \\ \sigma(y_1) & \sigma(y_2)\end{pmatrix}\begin{pmatrix} \alpha_\gamma & \beta_\gamma \\ 0 & \lambda_\gamma\end{pmatrix}=\begin{pmatrix} \alpha_\gamma y_1 & \beta_\gamma y_1 + \lambda_\gamma y_2 \\ \alpha_\gamma\sigma(y_1) & \beta_\gamma\sigma(y_1)+\lambda_\gamma\sigma(y_2)\end{pmatrix}\end{equation} for some $\alpha_\gamma,\beta_\gamma,\lambda_\gamma\in C$ such that $\alpha_\gamma\lambda_\gamma\neq 0$. Since $$\gamma\left(\frac{\sigma(y_1)}{y_1}\right)=\frac{\alpha_\gamma\sigma(y_1)}{\alpha_\gamma y_1}=\frac{\sigma(y_1)}{y_1}$$ for every $\gamma\in G$, it follows from Theorem~\ref{correspondence} that $\frac{\sigma(y_1)}{y_1}=u\in k$.

Since $G$ is Zariski-dense in the $\sigma$-Galois group $H$ associated to \eqref{relations-eq}, the reducibility of $G$ implies that of $H$. As we discussed in \S\ref{hendriks-sec}, it was already proved in \cite{hendriks:1998} that $H$ is reducible but not diagonalizable if and only if there is a solution $y_1$ for \eqref{relations-eq} that satisfies $\sigma(y_1)=uy_1$ for some $u\in\bar{\mathbb{Q}}(x)$, and the coefficient $u$ is obtained as the unique solution in $\bar{\mathbb{Q}}(x)$ to the Riccati equation \begin{equation}\label{rel-ric1}u\sigma(u)+au+b=0.\end{equation} As we discussed in \S\ref{reducible-sec}, the existence of $u\in\bar{\mathbb{Q}}(x)$ satisfying \eqref{rel-ric1} implies the factorization of the operator implicit in \eqref{relations-eq} \begin{equation}\label{relations-factorization} \sigma^2+a\sigma+b=(\sigma-\tfrac{b}{u})\circ(\sigma-u),\end{equation} which implies that the element $0\neq y_0=\sigma(y_2)-uy_2$ satisfies $\sigma(y_0)=\frac{b}{u}y_0$. Hence, we can compute all the $\delta$-algebraic relations satisfied by $y_1$ and $y_0$ over $k$ as in \S\ref{diag-subsec}, after replacing $v$ with $\frac{b}{u}$ and $y_2$ with $y_0$. 

In order to simplify the discussion, let us make the supplementary assumption that $H$ and $G$ are connected. As we explained in Remark~\ref{connected-rem}, this is sufficient to deduce all the $\sigma\delta$-algebraic relations satisfied by $y_1$ and $y_2$ over $k$. Since $R_u(G)$ is the $\sigma\delta$-Galois group for $\sigma(y_2)-uy_2=y_0$ over the field of fractions $L$ of the $\sigma\delta$-PV ring $k\{y_1,y_0, (y_1y_0)^{-1}\}_\delta$ for \eqref{redeq} over $k$, it follows from \cite[Prop.~6.26]{hardouin-singer:2008} that if $R_u(G)=\mathbb{G}_a(C)$, then the transcendence degree of $y_2$ over $L$ is $1$, and therefore $y_2$ is $\delta$-transcendental over $L$ whenever $R_u(G)=\mathbb{G}_a(C)$, in which case the only $\sigma\delta$-algebraic relations among $y_1$ and $y_2$ are all $\delta$-algebraic consequences of the relations satisfied by $y_1$ and $y_0=\sigma(y_2)-uy_2$ over $k$.

It follows from Proposition~\ref{bigut2} and Proposition~\ref{wrat} that the only case where $R_u(G)\neq\mathbb{G}_a(C)$ is when $G$ is one of the groups described in Proposition~\ref{wrat}(iii). So let us assume that $\alpha_\gamma=\lambda_\gamma$ for every $\gamma\in G$ and there exists a nonzero linear $\delta$-polynomial $\mathcal{L}\in\bar{\mathbb{Q}}\{Y\}_\delta$ such that $\beta_\gamma=\alpha_\gamma\mathcal{L}(\frac{\delta(\alpha_\gamma)}{\alpha_\gamma})$ for every $\gamma\in G$. Then from \eqref{red-emb} we obtain $$\gamma\left(\frac{y_2}{y_1}\right)=\frac{\lambda_\gamma y_2+\beta_\gamma y_1}{\alpha_\gamma y_1}=\frac{\alpha_\gamma y_2+\alpha_\gamma\mathcal{L}(\tfrac{\delta(\alpha_\gamma)}{\alpha_\gamma})y_1}{\alpha_\gamma y_1}=\frac{y_2}{y_1}+\mathcal{L}\left(\frac{\delta(\alpha_\gamma)}{\alpha_\gamma}\right).$$ On the other hand, $$\gamma\left(\mathcal{L}\left(\frac{\delta(y_1)}{y_1}\right)\right)=\mathcal{L}\left(\frac{\delta(\alpha_\gamma y_1)}{\alpha_\gamma y_1}\right)=\mathcal{L}\left(\frac{\delta(y_1)}{y_1}\right)+\mathcal{L}\left(\frac{\delta(\alpha_\gamma)}{\alpha_\gamma}\right).$$ Therefore, $$\gamma\left(\frac{y_2}{y_1}-\mathcal{L}\left(\frac{\delta(y_1)}{y_1}\right)\right)=\frac{y_2}{y_1}-\mathcal{L}\left(\frac{\delta(y_1)}{y_1}\right)=g\in k.$$ Thus, from our assumptions on the defining equations for the Galois group we have obtained the existence of a rational function $g\in k$ such that $y_2=y_1\mathcal{L}(\frac{\delta(y_1)}{y_1})+gy_1$. To compute the coefficient $g$, we take this expression for $y_2$ as an ansatz and proceed as follows. Recall that the existence of $u\in\bar{\mathbb{Q}}(x)$ satisfying \eqref{rel-ric1} is equivalent to the factorization \eqref{relations-factorization}. Now we compute \begin{align*}
0 &=\sigma^2(y_2)+a\sigma(y_2)+by_2 \\ 
&=(\sigma-\tfrac{b}{u})\circ(\sigma-u)\left(y_1\mathcal{L}\left(\frac{\delta(y_1)}{y_1}\right)+gy_1\right)\\
&=(\sigma-\tfrac{b}{u})\left[uy_1\mathcal{L}\left(\frac{\delta(y_1)}{y_1}\right)+uy_1\mathcal{L}\left(\frac{\delta(u)}{u}\right)+u\sigma(g)y_1-uy_1\mathcal{L}\left(\frac{\delta(y_1)}{y_1}\right)-ugy_1\right] \\
&=(\sigma-\tfrac{b}{u})\left[\left( \mathcal{L}\left(\frac{\delta(u)}{u}\right) + \sigma(g) - g\right)uy_1\right] \\
&= \left[u\sigma(u)\mathcal{L}\left(\frac{\delta(\sigma(u))}{\sigma(u)}\right)+u\sigma(u)\sigma^2(g)-u\sigma(u)\sigma(g)-b\mathcal{L}\left(\frac{\delta(u)}{u}\right)-b\sigma(g)+bg\right]y_1,\end{align*} and therefore $g\in k$ must satisfy the following linear inhomogeneous equation (where we have used the fact that $au=-u\sigma(u)-b$ in obtaining the coefficient of $\sigma(g)$ below):\begin{equation}\label{f-inhom} [u\sigma(u)]\sigma^2(g)+[au]\sigma(g)+[b]g=\left[b\mathcal{L}\left(\frac{\delta(u)}{u}\right)-u\sigma(u)\mathcal{L}\left(\frac{\delta(\sigma(u))}{\sigma(u)}\right)\right].\end{equation} Since all the coefficients of \eqref{f-inhom} belong to $\bar{\mathbb{Q}}(x)$, we may take $g\in \bar{\mathbb{Q}}(x)$ by Remark~\ref{c-ok}. We remark that this coefficient $g$ may be computed in a different way, following the proof of Proposition~\ref{wrat}(iii): if $w\in\bar{\mathbb{Q}}(x)$ is a solution to $\sigma(w)=\frac{b}{u\sigma(u)}w$, then $g\in\bar{\mathbb{Q}}(x)$ also satisfies $\sigma(g)-g=w-\mathcal{L}(\frac{\delta(u)}{u})$.

\subsection{$G$ is irreducible and imprimitive}

Let us now assume that $G$ is irreducible and imprimitive as in \S\ref{dihedral-sec}, so that $G$ is one of the subgroups of \begin{equation}\label{rel-dih-gp}\{\pm1\}\ltimes\mathbb{G}_m(C)^2= \left\{\begin{pmatrix} \alpha & 0 \\ 0 & \lambda\end{pmatrix} \ \middle| \ \alpha,\lambda\in C, \ \alpha\lambda\neq 0\right\} \cup\left\{\begin{pmatrix} 0 & \beta \\ \epsilon & 0\end{pmatrix} \ \middle| \ \beta,\epsilon\in C, \ \beta\epsilon\neq 0\right\}\end{equation} described in Proposition~\ref{dihedral-complete}. By \cite[Lem.~4.5]{hendriks:1998}, we know that \eqref{relations-eq} is equivalent to \begin{equation}\label{relations-dih-eq}\sigma^2(y)+ry=0.\end{equation} By \cite[Cor.~4.3]{hendriks-singer:1999}, all the solutions of \eqref{relations-dih-eq} can be expressed as the \emph{interlacing} \cite[Def.~3.2]{hendriks-singer:1999} of two hypergeometric elements. We refer to \cite[\S4]{hendriks-singer:1999} for a discussion of this phenomenon, and limit ourselves to deducing the relations satisfied by the solutions in the cases described in Proposition~\ref{dihedral-complete}.

Suppose that the matrix $T_\gamma\in\mathrm{GL}_2(C)$ associated to each $\gamma\in G$ by the choice of fundamental solution matrix $Z$ as in \eqref{rel-fund} belongs to $\{\pm 1\}\ltimes\mathbb{G}_m(C)^2$ as in \eqref{rel-dih-gp}. We will see from the explicit form of the conditions described in Proposition~\ref{dihedral-complete} that the $\sigma\delta$-algebraic relations among $y_1$ and $y_2$ are essentially captured by the \emph{Casoratian}: $$\omega=y_1\sigma(y_2)-y_2\sigma(y_1)=\mathrm{det}(Z),$$ which satisfies $\sigma(\omega)=r\omega$, so $\omega$ is hypergeometric. We remark that, since $y_1$ and $y_2$ are $C$-linearly independent, $\omega\neq 0$ by \cite[Lem.~A.2]{hendriks-singer:1999}; moreover, it is shown in \cite[proof of Lem.~A.6]{hendriks-singer:1999} that $\omega$ is invertible. We also know that $\gamma(\omega)=\mathrm{det}(T_\gamma)\omega$ for each $\gamma\in G$.

Let us first show that $y_1y_2=0$. We proceed by contradiction: assuming that $y_1y_2\neq 0$, we will show that there exists a solution $u\in\bar{\mathbb{Q}}(x)$ to the Riccati equation $u\sigma(u)+r=0$, contradicting the assumption that $H$ is irreducible by the results of \cite{hendriks:1998} summarized in \S\ref{hendriks-sec}. For each $\gamma\in G$ we see that either $\gamma(y_1)=\alpha_\gamma y_1$ and $\gamma(y_2)=\lambda_\gamma y_2$, or else $\gamma(y_1)=\epsilon_\gamma y_2$ and $\gamma(y_2)=\beta_\gamma y_1$. In any case, it follows that $\gamma(y_1y_2)=\pm\mathrm{det}(T_\gamma)y_1y_2$, and therefore $\gamma(\frac{y_1y_2}{\omega})=\pm\frac{y_1y_2}{\omega}$, whence $\gamma(\frac{y_1^2y_2^2}{\omega^2})=\frac{y_1^2y_2^2}{\omega^2}$ for every $\gamma\in G$. It follows from Theorem~\ref{correspondence} that there exists $g\in k$ such that $y_1^2y_2^2=g\omega^2$. If $y_1y_2\neq 0$, then $0\neq g\in k$ is invertible. Now we have that $$\sigma^2\left(\frac{y_1y_2}{\omega}\right)=\frac{(-ry_1)(-ry_2)}{r\sigma(r)\omega}=\frac{r}{\sigma(r)}\cdot\frac{y_1y_2}{\omega},\quad \text{and therefore} \quad\sigma^2(g)=\sigma^2\left(\frac{y_1^2y_2^2}{\omega^2}\right)=\frac{r^2}{\sigma(r)^2}\cdot\frac{y_1^2y_2^2}{\omega^2}=\frac{r^2}{\sigma(r)^2}\cdot g.$$ Since this latter equation has coefficients in $\bar{\mathbb{Q}}(x)$, we may take $g\in\bar{\mathbb{Q}}(x)$ by \cite[Lem.~2.5]{hardouin:2008} (cf.~Remark~\ref{c-ok}), and it follows from $$\frac{\sigma(g\sigma(g))}{g\sigma(g)}=\frac{\sigma^2(g)}{g}=\frac{r^2}{\sigma(r)^2}=\frac{\sigma(r^{-2})}{r^{-2}}$$ that $g\sigma(g)=cr^{-2}$ for some $c\in \bar{\mathbb{Q}}^\times$. Let us assume without loss of generality that $c=1$. There exist rational functions $g_1,g_2\in \bar{\mathbb{Q}}(x)$, unique up to constant multiple, such that $g=g_1g_2^2$ and $g_1$ is squarefree. We claim that $g_1$ is constant: otherwise, there would exist a zero (resp., pole) $d\in \bar{\mathbb{Q}}$ of $g_1$ such that $d+1$ is not a zero (resp., pole) of $g_1$, and we would have that $\pm 1=\mathrm{ord}_d(g_1\sigma(g_1))=\mathrm{ord}_d(r^{-2}g_2^{-2}\sigma(g_2^{-2}))\in 2\mathbb{Z}$, a contradiction. So we may take $g=g_2^2$ without loss of generality, and it follows that $\pm(g_2\sigma(g_2))^{-1}=r$, whence either $u=(g_2)^{-1}$ or $u=(\sqrt{-1} g_2)^{-1}$ satisfies the Riccati equation $u\sigma(u)+r=0$. This contradiction concludes the proof that $y_1y_2=0$.

In each one of the cases described in Proposition~\ref{dihedral-complete}(i), we see that $G$ is the subgroup of \eqref{rel-dih-gp} defined by the condition $(\mathrm{det}(T_\gamma))^m=1$, where $m$ is a positive integer. Therefore, $\gamma(\omega^m)=(\mathrm{det}(T_\gamma))^m\omega^m=\omega^m$ for each $\gamma\in G$, so it follows from Theorem~\ref{correspondence} that $(y_1\sigma(y_2)-y_2\sigma(y_1))^m=f\in k$, and we obtain $f$ from the relation $\sigma(f)=r^mf$, so we may take $f\in\bar{\mathbb{Q}}(x)$ by \cite[Lem.~2.5]{hardouin:2008} (cf.~Remark~\ref{c-ok}).

If $\delta(\mathrm{det}(T_\gamma))=0$ for every $\gamma\in G$ as in Proposition~\ref{dihedral-complete}(ii), then $$\gamma\left(\frac{\delta(\omega)}{\omega}\right)  = \frac{\delta(\mathrm{det}(T_\gamma)\omega)}{\mathrm{det}(T_\gamma)\omega}= \frac{\delta(\omega)}{\omega}+\frac{\delta(\mathrm{det}(T_\gamma))}{\mathrm{det}(T_\gamma)}=\frac{\delta(\omega)}{\omega}.$$ It follows from Theorem~\ref{correspondence} that $\delta(\omega)=f\omega$ for some $f\in k$, and we obtain the coefficient $f$ from the relation $\sigma(f)-f=\frac{\delta(r)}{r}$, so we may take $f\in\bar{\mathbb{Q}}(x)$ by \cite[Lem.~2.4]{hardouin:2008} (cf.~Remark~\ref{c-ok}).

Finally, if $G=\{\pm 1\}\ltimes\mathbb{G}_m(C)^2$ as in Proposition~\ref{dihedral-complete}(iii), then all the $\sigma\delta$-algebraic relations among $y_1$ and $y_2$ are $\delta$-algebraic consequences of $y_1y_2=0$ and $\sigma(\omega)=r\omega$.

\subsection{$G$ contains $\mathrm{SL}_2(C)$}

Let us now assume that $G$ is neither reducible nor imprimitive as in \S\ref{large-sec}, so $G$ is one of the subgroups of $\mathrm{GL}_2(C)$ described in Theorem~\ref{g-large}. Let $T_\gamma\in\mathrm{GL}_2(C)$ denote the matrix associated to $\gamma\in G$ by a given choice of fundamental solution matrix $Z$ as in \eqref{rel-fund}. We will see from the explicit conditions described in Theorem~\ref{g-large} that the $\sigma\delta$-algebraic relations among $y_1$ and $y_2$ are completely captured by the \emph{Casoratian} $$\omega=y_1\sigma(y_2)-y_2\sigma(y_1)=\mathrm{det}(Z).$$ Note that $\sigma(\omega)=b\omega$, and $\gamma(\omega)=\mathrm{det}(T_\gamma)\omega$ for each $\gamma\in G$. We will consider each possibility from Theorem~\ref{g-large} in turn, and deduce the corresponding relations that are satisfied by the solutions to~\eqref{relations-eq} in each case.

If $G=\{T\in\mathrm{GL}_2(C) \ | \ \mathrm{det}(T)^m=1\}$ as in Theorem~\ref{g-large}(i), then $\gamma(\omega^m)=(\mathrm{det}(T_\gamma))^m\omega^m=\omega^m$ for each $\gamma\in G$, so it follows from Theorem~\ref{correspondence} that $(y_1\sigma(y_2)-y_2\sigma(y_1))^m=f\in k$, and we obtain $f$ from the relation $\sigma(f)=b^mf$, so we may take $f\in\bar{\mathbb{Q}}(x)$ by \cite[Lem.~2.5]{hardouin:2008} (cf.~Remark~\ref{c-ok}).

If $G=\{T\in\mathrm{GL}_2(C) \ | \ \delta(\mathrm{det}(T))=0\}$ as in Theorem~\ref{g-large}, then $$\gamma\left(\frac{\delta(\omega)}{\omega}\right)  = \frac{\delta(\mathrm{det}(T_\gamma)\omega)}{\mathrm{det}(T_\gamma)\omega}= \frac{\delta(\omega)}{\omega}+\frac{\delta(\mathrm{det}(T_\gamma))}{\mathrm{det}(T_\gamma)}=\frac{\delta(\omega)}{\omega}.$$ It follows from Theorem~\ref{correspondence} that $\delta(\omega)=f\omega$ for some $f\in k$, and we obtain the coefficient $f$ from the relation $\sigma(f)-f=\frac{\delta(b)}{b}$, so we may take $f\in\bar{\mathbb{Q}}(x)$ by \cite[Lem.~2.4]{hardouin:2008} (cf.~Remark~\ref{c-ok}).

Finally, if $G=\mathrm{GL}_2(C)$ as in Theorem~\ref{g-large}(iii), then $y_1$, $y_2$, $\sigma(y_1)$, and $\sigma(y_2)$ are $\delta$-independent over $k$.

\section{Examples} \label{examples-sec}
In this section we compute the $\sigma\delta$-Galois group $G$ associated to some concrete second-order linear difference equations over $\bar{\mathbb{Q}}(x)$ with respect to the shift operator $\sigma:x\mapsto x+1$. We will first apply the algorithm of \cite{hendriks:1998} to compute the $\sigma$-Galois group $H$ associated to the equation, and then apply the procedures developed in this paper to compute $G$. Our computations are performed using the procedures ratpolysols and hypergeomsols of the Maple package LREtools.

\subsection{Example}
Let us consider \eqref{difeq} with $a=-(2x+1)$ and $b=x^2$: \begin{equation}\label{eg1}\sigma^2(y)-(2x+1)\sigma(y)+x^2y=0.\end{equation} The Maple procedure LREtools[hypergeomsols] shows that the only hypergeometric solutions to \eqref{eg1} are of the form $c\Gamma(x)$, where $c\in\bar{\mathbb{Q}}$ and the $\Gamma$ function satisfies $\sigma(\Gamma)=x\Gamma$. Therefore, the Riccati equation \eqref{ric1} admits a unique solution $u=x$ in $\bar{\mathbb{Q}}(x)$. We begin by computing the reductive quotient $G/R_u(G)$ by applying Proposition~\ref{galdiag} to the system $$\sigma(Y)=\begin{pmatrix} u & 0 \\ 0 & b/u\end{pmatrix}Y=\begin{pmatrix} x & 0 \\ 0 & x\end{pmatrix}Y,$$ we find that cases (i) and (iii) of Proposition~\ref{galdiag} do not hold, because there is no $f\in\bar{\mathbb{Q}}(x)$ such that $\frac{\delta(u)}{u}=\frac{1}{x}=\sigma(f)-f$. However, Proposition~\ref{galdiag}(ii) does hold with $m=1$, $n=-1$, and $f=1$. Hence, the $\sigma$-Galois group $H$ for \eqref{eg1} is $$H=\left\{\begin{pmatrix} \alpha & \beta \\ 0 & \alpha\end{pmatrix} \ \middle| \ \alpha,\beta\in C, \ \alpha\neq 0\right\}.$$ To compute $G$, we apply Proposition~\ref{wrat}: we find that $w\in\bar{\mathbb{Q}}(x)$ must satisfy $\sigma(w)=\frac{b}{u\sigma(u)}=\frac{x}{x+1}w$, so $w=\frac{1}{x}$. To test whether the conditions of Proposition~\ref{wrat}(iii) hold, we attempt to find an operator $\mathcal{L}\in\bar{\mathbb{Q}}[\delta]$, of smallest possible order, such that there exists $g\in\bar{\mathbb{Q}}(x)$ with $$\mathcal{L}\left(\frac{\delta(u)}{u}\right)-w=\mathcal{L}\left(\frac{1}{x}\right)-\frac{1}{x}=\sigma(g)-g.$$ And we find that $\mathcal{L}=1$ and $g=0$ satisfy these conditions, and therefore the $\sigma\delta$-Galois group for \eqref{eg1} is $$G=\left\{\begin{pmatrix} \alpha & \delta(\alpha) \\ 0 & \alpha\end{pmatrix} \ \middle| \ \alpha\in C, \ \alpha\neq 0 \right\}.$$

\subsection{Example}
Let us consider \eqref{difeq} with $a=0$ and $b=\frac{x+1}{2x}$: \begin{equation} \label{eg2} \sigma^2(y)+\tfrac{x+1}{2x}y=0.\end{equation} it follows from \cite[Lem.~4.5]{hendriks:1998} that the $\sigma$-Galois group $H$ associated to \eqref{eg2} is a subgroup of \eqref{dihedral}. To first check whether $H$ is irreducible by deciding whether \eqref{eg2} admits any hypergeometric solutions, which is equivalent to the existence of a solution $u\in\bar{\mathbb{Q}}(x)$ to the Riccati equation \eqref{ric1}. The Maple procedure LREtools[hypergeomsols] returns $0$ as the only hypergeometric solution to \eqref{eg2}, and therefore $H$ is irreducible and imprimitive. We now check the conditions of Proposition~\ref{dihedral-complete} with $r=\frac{x+1}{2x}=\frac{1}{2}\frac{\sigma(x)}{x}$, and find that case (i) does not hold, which shows that $H$ coincides with \eqref{dihedral}. However, Proposition~\ref{dihedral-complete}(iii) does hold, with $f=\frac{1}{x}$, and therefore the $\sigma\delta$-Galois group $G$ associated to \eqref{eg2} is $$G=\left\{\begin{pmatrix} \alpha & 0 \\ 0 & \lambda\end{pmatrix} \ \middle| \ \alpha,\lambda\in C, \ \alpha\lambda\neq 0, \ \delta(\alpha\lambda)=0\right\} \cup\left\{\begin{pmatrix} 0 & \beta \\ \epsilon & 0\end{pmatrix} \ \middle| \ \beta,\epsilon\in C, \ \beta\epsilon\neq 0, \ \delta(\beta\epsilon)=0\right\}.$$

\subsection{Example}
Let us consider \eqref{difeq} with $a=x$ and $b=1$: \begin{equation}\label{eg3} \sigma^2(y)+x\sigma(y)+y=0.\end{equation} It is proved in \cite[proof of Lem.~3.9]{vanderput-singer:1997} that the $\sigma$-Galois group $H$ associated to \eqref{eg3} is $\mathrm{SL}_2(C)$. To see this directly, we first apply the Maple procedure LREtools[hypergeomsols] to verify that \eqref{eg3} does not admit any hypergeometric solutions. This shows that $H$ is irreducible. To verify that $H$ is primitive, we have to verify that the Riccati equation \eqref{ric2} does not admit any solutions $e\in\bar{\mathbb{Q}}(x)$. By \cite[Rem.~4.7]{hendriks:1998}, this is equivalent to the statement that the difference equation $$\sigma^2(y)+\frac{1-6x^2-4x^3}{2x^2+2x}\sigma(y)+\frac{1}{4x^2}y=0$$ has no hypergeometric solutions, which we again verify with the Maple procedure LREtools[hypergeomsols]. This implies that $\mathrm{SL}_2(C)\subseteq H$. Since $1=b=\frac{\sigma(1)}{1}$, it follows from Theorem~\ref{g-large}(i) that $G=\mathrm{SL}_2(C)$.

\bibliographystyle{spmpsci} \bibliography{difference-books}{} \nocite{*}

\end{document}